\documentclass{amsart}

\makeatletter
\@namedef{subjclassname@2020}{%
  \textup{2020} Mathematics Subject Classification}
\makeatother 

\usepackage{amsmath,amscd}
\usepackage{amssymb,latexsym,amsthm}
\usepackage[spanish,english]{babel}
\usepackage{amssymb}
\usepackage{color}
\usepackage{amsmath,amsthm,amscd}
\usepackage{lscape}
\usepackage{fancyhdr}
\usepackage{amsfonts}
\usepackage{pb-diagram}
\usepackage{soul}
\usepackage{cancel}
\usepackage[all]{xy} 
\usepackage[utf8]{inputenc}
\usepackage[dvipsnames]{xcolor}
\usepackage{hyperref} 
\hypersetup{
    colorlinks=true,
    linkcolor=blue,
    filecolor=blue,      
    urlcolor=Violet,
    linktocpage=true
}

\newtheorem{theorem}{Theorem}[section]
\newtheorem{lemma}[theorem]{Lemma}

\theoremstyle{definition}
\newtheorem{definition}[theorem]{Definition}
\newtheorem{proposition}[theorem]{Proposition}
\newtheorem{example}[theorem]{Example}
\newtheorem{remark}[theorem]{Remark}
\newtheorem{corollary}[theorem]{Corollary}

\numberwithin{equation}{section}

\begin{document}

\title[Uniform dimension and associated primes of skew PBW extensions]{On the uniform dimension and the associated primes of skew PBW extensions}


\author{Sebasti\'an Higuera}
\address{Universidad Nacional de Colombia - Sede Bogot\'a}
\curraddr{Campus Universitario}
\email{sdhiguerar@unal.edu.co}

\author{Mar\'ia Camila Ram\'irez}
\address{Universidad Nacional de Colombia - Sede Bogot\'a}
\curraddr{Campus Universitario}
\email{macramirezcu@unal.edu.co}

\author{Armando Reyes}
\address{Universidad Nacional de Colombia - Sede Bogot\'a}
\curraddr{Campus Universitario}
\email{mareyesv@unal.edu.co}
\thanks{}


\thanks{The authors were supported by the research fund of Faculty of Science, Code HERMES 53880, Universidad Nacional de Colombia - Sede Bogot\'a, Colombia}

\subjclass[2020]{16D25, 16P60, 16S36, 16S38}

\keywords{Induced module, associated prime, uniform dimension, skew PBW extension.}

\date{}

\dedicatory{Dedicated to Professor Oswaldo Lezama}

\begin{abstract}
In this paper, we study the uniform dimension and the associated prime ideals of induced modules over skew PBW extensions.
\end{abstract}

\maketitle


\section{Introduction}\label{Introduccion}

Throughout the paper, every ring $R$ is associative (not necessarily commutative) with identity unless stated otherwise. In addition, $S$ denotes the {\em skew polynomial ring} of $R$ (also known as {\em Ore extension}) $R[x;\sigma,\delta]$ defined by Ore \cite{Ore1933} where $\sigma$ is an automorphism of $R$ and $\delta$ is a $\sigma$-derivation of $R$. In the context of the skew polynomial rings, Leroy and Matczuk \cite{LeroyMatczuk2004} considered the {\em induced modules} on these noncommutative rings: if $M_R$ is a right module and $S$ is a skew polynomial ring of $R$, then $M \otimes_R S := \widehat{M}_S $ is said to be the {\it induced module} of $M_R$ \cite[p. 2745]{LeroyMatczuk2004}. They investigated problems related to the uniform dimension and the associated primes of the induced module $\widehat{M}_S$ by considering {\em good polynomials}.

These polynomials were used by Shock with the aim of proving that the uniform dimensions of a ring $R$ and the polynomial ring $R[x]$ are equal  \cite[Theorem 2.6]{Shock1972}. A polynomial $f(x) \in R[x]$ is called {\em good} if the annihilators of the coefficients of $f(x)$ are equal \cite[p. 252]{Shock1972}. A polynomial $g \in \widehat{M}_S$ is called {\em good} if for any $r \in R$, ${\rm deg}(gr) = {\rm deg}(g)$ provided $gr \neq 0$ \cite[Definition 3.2]{LeroyMatczuk2004}. A submodule $B_S$ of $\widehat{M}_S$ is called {\it good}, if for any good polynomial $g \in B_S$ and any $n \geq {\rm deg(g)}$, there exists a good polynomial of degree $n$ in $gS$ \cite[Definition 4.4]{LeroyMatczuk2004}. Leroy and Matczuk characterized the good polynomials of $\widehat{M}_S$ and the right annihilators of generated modules by these polynomials \cite[Lemmas 3.4 and 3.7]{LeroyMatczuk2004}. In addition, they described the essential and uniform submodules of the module $\widehat{M}_S$ \cite[Theorem 4.6]{LeroyMatczuk2004}, and proved that if $\widehat{M}_S$ is a good module, then the uniform dimensions of $M_R$ and $\widehat{M}_S$ are the same \cite[Theorem 4.9]{LeroyMatczuk2004}. They showed that all associated prime ideals of the induced module $\widehat{M}_S$ arise from associated primes of the module $M_R$ \cite[Theorem 5.7]{LeroyMatczuk2004}. Continuing with the study of associated prime ideals and as a natural generalization of these ideals, Ouyang and Birkenmeier \cite{OuyangBirkenmeier2012} defined the {\em nilpotent associated primes} \cite[Definition 3.2]{OuyangBirkenmeier2012} and considered the {\em nilpotent good polynomials} as a tool to investigate these ideals \cite[Definition 3.3]{OuyangBirkenmeier2012}. They characterized the nilpotent associated primes of skew polynomial rings \cite[Theorem 3.1]{OuyangBirkenmeier2012}.

Gallego and Lezama \cite{GallegoLezama2011} defined the {\em skew PBW extensions} as a generalization of the Poincar\'e-Birkhoff-Witt extensions introduced by Bell and Goodearl \cite{BellGoodearl1988} and the skew polynomial rings of injective type. Since its introduction, ring and homological properties of skew PBW extensions have been widely studied. In the literature, several authors have shown that the skew PBW extensions generalize families of noncommutative algebras such as 3-dimensional skew polynomial algebras introduced by Bell and Smith \cite{BellSmith1990}, ambiskew polynomial rings in the sense of Jordan \cite{Jordan2000}, solvable polynomial rings by Kandri-Rody and Weispfenning  \cite{KandryWeispfenninig1990}, almost normalizing extensions defined by McConnell and Robson \cite{McConnellRobson2001}, and skew bi-quadratic algebras recently introduced by Bavula \cite{Bavula2021}. For more details about skew PBW extensions and other noncommutative algebras having PBW bases, see \cite{AbdiTalebi2024, LFGRSV, GomezTorrecillas2014, Seiler2010}. Related to the good polynomials and the associated primes, Higuera and Reyes \cite{HigueraReyes2022} extended the notion of nilpotent good polynomials and characterized the nilpotent associated prime ideals over skew PBW extensions \cite[Theorem 4.4]{HigueraReyes2022}. Annin \cite{Annin2004} considered the {\em annihilator-compliant polynomials} to investigate associated prime ideals of the induced module $\widehat{M}_S$. According to Annin, $m(x) = m_0 + \cdots + m_kx^k \in M[x]$ with $m_k\neq 0$ is called {\em annihilator-compliant} if for each $i < k$, ${\rm ann}_R(m_k) \subseteq {\rm ann}_R(m_i)$ \cite[Definition 2.23]{Annin2004}. Ni\~no et al. \cite{NinoRamirezReyes} extended this definition to study associated prime ideals of modules over skew PBW extensions. Under certain compatibility conditions, they characterized the associated primes of the induced module $M\otimes_R A:= M\langle X \rangle_A$ where $A$ is a skew PBW extension over a ring $R$ \cite[Theorem 3.12]{NinoRamirezReyes}.

Thinking about the above results and motivated for the development of the theory of the uniform dimension and the associated prime ideals of polynomial modules over noncommutative rings of polynomial type (see \cite{Annin2004}, \cite{LeroyMatczuk2004}, \cite{NinoRamirezReyes}, \cite{HigueraReyes2022}, and references therein), our aim in this paper is to characterize essential modules and the uniform dimension of induced modules over skew PBW extensions. Additionally, we study the uniform dimension of induced modules and investigate the associated prime ideals of induced modules over families of rings more general than skew polynomial rings.

The paper is organized as follows. In Section \ref{Preliminares}, we recall some definitions and preliminaries about skew PBW extensions. Section \ref{Goodpolynomials} presents the definition of good polynomial and original results that characterize these polynomials (Lemmas \ref{PBWLemma3.4}, \ref{PBWCorollary3.5},  \ref{PBWLemma3.7}, and Proposition \ref{PBWProposition3.6}). We also present several results on essential modules and uniform dimension, within which we characterize the uniform dimension of induced modules over skew PBW extensions (Lemma \ref{PBWLemma4.1}, and Theorems \ref{PBWTheorem4.6} and \ref{PBWTheorem4.9}). Section \ref{Associatedprimes} contains results related to the characterization of associated primes of induced modules over these extensions (Lemma \ref{PBWLemma5.4}, and Theorems \ref{PBWTheorem5.7} and \ref{PBWTheorem5.10}). Our results generalize those corresponding presented by Leroy and Matczuk \cite{LeroyMatczuk2004}. It is worth mentioning that this work is a sequel of the study of ideals of skew PBW extensions that has been realized by different authors (e.g. \cite{ LezamaAcostaReyes2015, NinoReyes2019, ReyesSuarez2019CMS, ReyesSuarezYesica2018}). In this way, the results formulated in this paper about associated prime ideals extend or contribute to those presented by Annin \cite{Annin2004}, Brewer and Heinzer \cite{BrewerHeinzer1974}, Faith \cite{Faith2000}, Leroy and Matczuk \cite{LeroyMatczuk2004}, Ni\~no et al., \cite{NinoRamirezReyes}, and references therein. Finally, Section \ref{examplespaper} illustrates the results established in Sections \ref{Goodpolynomials} and \ref{Associatedprimes} with several noncommutative algebras that cannot be expressed as skew polynomial rings.

Throughout the paper, $\mathbb{N}$, $\mathbb{Z}$, $\mathbb{R}$, and $\mathbb{C}$ denote the classical numerical systems. We assume the set of natural numbers including zero. The symbol $\Bbbk$ denotes a field and $\Bbbk^{*} := \Bbbk\ \backslash\ \{0\}$.

\section{Skew Poincar\'e-Birkhoff-Witt extensions}\label{Preliminares}

\begin{definition}[{\cite[Definition 1]{GallegoLezama2011}}] \label{def.skewpbwextensions}
Let $R$ be a ring. A ring $A$ is said to be a \textit{skew PBW extension over} $R$ (the ring of coefficients), denoted $A=\sigma(R)\langle
x_1,\dots,x_n\rangle$, if the following conditions hold:
\begin{enumerate}
\item[\rm (i)]$R$ is a subring of $A$ sharing the same identity element.

\item[\rm (ii)] There exist finitely many elements $x_1,\dots ,x_n\in A$ such that $A$ is a left free $R$-module, with basis the
set of standard monomials
\begin{center}
${\rm Mon}(A):= \{x^{\alpha}:=x_1^{\alpha_1}\cdots
x_n^{\alpha_n}\mid \alpha=(\alpha_1,\dots ,\alpha_n)\in
\mathbb{N}^n\}$.
\end{center}
Moreover, $x^0_1\cdots x^0_n := 1 \in {\rm Mon}(A)$.

\item[\rm (iii)]For every $1\leq i\leq n$ and any $r\in R\ \backslash\ \{0\}$, there exists $c_{i,r}\in R\ \backslash\ \{0\}$ such that $x_ir-c_{i,r}x_i\in R$.

\item[\rm (iv)]For $1\leq i,j\leq n$, there exists $d_{i,j}\in R\ \backslash\ \{0\}$ such that
\[
x_jx_i-d_{i,j}x_ix_j\in R+Rx_1+\cdots +Rx_n,
\]
i.e. there exist elements $r_0^{(i,j)}, r_1^{(i,j)}, \dotsc, r_n^{(i,j)} \in R$ with
\begin{center}
$x_jx_i - d_{i,j}x_ix_j = r_0^{(i,j)} + \sum_{k=1}^{n} r_k^{(i,j)}x_k$.    
\end{center}
\end{enumerate}
\end{definition}

Since ${\rm Mon}(A)$ is a left $R$-basis of $A$, the elements $c_{i,r}$ and $d_{i, j}$ are unique. Thus, every non-zero element $f \in A$ can be uniquely expressed as $f = \sum_{i=0}^ma_iX_i$, with $a_i \in R$, $X_0=1$, and $X_i \in \text{Mon}(A)$, for $0 \leq i \leq m$ \cite[Remark 2]{GallegoLezama2011}. 

\begin{proposition}[{\cite[Proposition 3]{GallegoLezama2011}}] \label{sigmadefinition}
If $A$ is a skew PBW extension over $R$, then there exist an injective endomorphism $\sigma_i:R\rightarrow R$ and a $\sigma_i$-derivation $\delta_i:R\rightarrow R$ such that $x_ir=\sigma_i(r)x_i+\delta_i(r)$, for each $1\leq i\leq n$, where $r\in R$.
\end{proposition}

We use the notation $\Sigma:=\{\sigma_1,\dots,\sigma_n\}$ and $\Delta:=\{\delta_1,\dots,\delta_n\}$ for the families of injective endomorphisms and derivations of Proposition \ref{sigmadefinition}, respectively.

\begin{definition}\label{quasicommutative}
Let $A$ be a skew PBW extension over $R$.
\begin{itemize}
    \item[{\rm (i)}] \cite[Definition 4]{GallegoLezama2011} $A$ is called {\it quasi-commutative} if the conditions ${\rm (iii)}$ and ${\rm (iv)}$ presented above are replaced by the following: 
\begin{enumerate}
    \item[(iii')] For every $1 \leq i \leq n$ and $r \in R \setminus \left \{0 \right \}$, there exists $c_{i,r} \in R \setminus \left \{0 \right \}$ such that $x_ir = c_{i,r}x_i$.
\item[(iv')] For every $1 \leq i, j \leq n$, there exists $d_{i,j} \in R \setminus \left \{0 \right \}$ such that 
\begin{center}
$x_jx_i = d_{i,j}x_ix_j$.    
\end{center}
\end{enumerate}
    \item[{\rm (ii)}] \cite[Definition 4]{GallegoLezama2011} $A$ is called {\it bijective} if  $\sigma_i$ is bijective for each $1 \leq i \leq n$, and $d_{i,j}$ is invertible for any $1 \leq i <j \leq n$.
    \item [\rm (iii)] \cite[Definition 2.3]{LezamaAcostaReyes2015} If $\sigma_i$ is the identity homomorphism of $R$ for all $1 \le i\le n$, then we say that $A$ is a skew PBW extension of \textit{derivation type}. Similarly, if $\delta_i\in \Delta$ is zero, for every $1 \le i\le n$, then $A$ is called a skew PBW extension of \textit{endomorphism type}.
\end{itemize}
\end{definition}

\begin{remark}\label{comparisonendomorphism}
Some relationships between skew polynomial rings and skew PBW extensions are the following:
\begin{itemize}
    \item[\rm (i)] If $A$ is a quasi-commutative skew PBW extension, then $A$ is isomorphic to an iterated skew polynomial ring of endomorphism type \cite[Theorem 2.3]{LezamaReyes2014}.
    
    \item[\rm (ii)] In general, skew polynomial rings of injective type are strictly contained in skew PBW extensions \cite[Example 5(3)]{LezamaReyes2014}. This fact is not possible for PBW extensions. For instance, the quantum plane $\Bbbk\{ x, y \} / \langle xy - qyx\mid q\in \Bbbk^{*} \rangle$ is a skew polynomial ring of injective type given by $\Bbbk[y][x;\sigma]$, where $\sigma(y) = qy$, but cannot be expressed as a PBW extension. 
    
   \item[\rm (iii)] Skew PBW extensions of endomorphism type are more general than iterated skew polynomial rings of endomorphism type \cite[Remark 2.4 (ii)]{SuarezChaconReyes2021}. 
\end{itemize}
\end{remark}

\begin{definition}[{\cite[Section 3]{GallegoLezama2011}}]\label{definitioncoefficients}
If $A$ is a skew PBW extension over $R$, then: 
\begin{enumerate}
\item[\rm (i)] For any element $\alpha=(\alpha_1,\dots,\alpha_n)\in \mathbb{N}^n$, we will write 
$\sigma^{\alpha}:=\sigma_1^{\alpha_1}\circ \dotsb \circ \sigma_n^{\alpha_n}$, $\delta^{\alpha} = \delta_1^{\alpha_1} \circ \dotsb \circ \delta_n^{\alpha_n}$, where $\circ$ denotes composition. If
$\beta=(\beta_1,\dots,\beta_n)\in \mathbb{N}^n$, then
$\alpha+\beta:=(\alpha_1+\beta_1,\dots,\alpha_n+\beta_n)$.

\item[\rm (ii)] Let $\succeq$ be a total order defined on ${\rm Mon}(A)$. If $x^{\alpha}\succeq x^{\beta}$ but $x^{\alpha}\neq x^{\beta}$, we write $x^{\alpha}\succ x^{\beta}$. If $f$ is a non-zero element of $A$, then we use expressions as $f=a_1x^{\alpha_1} + \cdots +a_kx^{\alpha_k}$, with $a_i\in R$, and $x^{\alpha_k}\succ \dotsb \succ x^{\alpha_1}$. With this notation, we define ${\rm
lm}(f):=x^{\alpha_k}$, the \textit{leading monomial} of $f$; ${\rm
lc}(f):=a_k$, the \textit{leading coefficient} of $f$; ${\rm
lt}(f):=a_kx^{\alpha_k}$, the \textit{leading term} of $f$. Note that $\deg(f):={\rm max}\{\deg(x^{\alpha_i})\}_{i=1}^k$. If $f=0$, ${\rm lm}(0):=0$, ${\rm lc}(0):=0$, ${\rm lt}(0):=0$. 
\end{enumerate}
\end{definition}

The next proposition is very useful when one need to make some computations with elements of skew PBW extensions.

\begin{proposition}[{\cite[Theorem 7]{GallegoLezama2011}}] \label{coefficientes}
If $A$ is a polynomial ring with coefficients in $R$ with respect to the set of indeterminates $\{x_1,\dots,x_n\}$, then $A$ is a skew PBW extension over $R$ if and only if the following conditions hold:
\begin{enumerate}
\item[\rm (1)]for each $x^{\alpha}\in {\rm Mon}(A)$ and every $0\neq r\in R$, there exist unique elements $r_{\alpha}:=\sigma^{\alpha}(r)\in R\ \backslash\ \{0\}$, $p_{\alpha ,r}\in A$, such that $x^{\alpha}r=r_{\alpha}x^{\alpha}+p_{\alpha, r}$, where $p_{\alpha ,r}=0$, or $\deg(p_{\alpha ,r})<|\alpha|$ if $p_{\alpha , r}\neq 0$. If $r$ is left invertible, so is $r_\alpha$.

\item[\rm (2)]For each $x^{\alpha},x^{\beta}\in {\rm Mon}(A)$, there exist unique elements $d_{\alpha,\beta}\in R$ and $p_{\alpha,\beta}\in A$ such that $x^{\alpha}x^{\beta} = d_{\alpha,\beta}x^{\alpha+\beta}+p_{\alpha,\beta}$, where $d_{\alpha,\beta}$ is left invertible, $p_{\alpha,\beta}=0$, or $\deg(p_{\alpha,\beta})<|\alpha+\beta|$ if
$p_{\alpha,\beta}\neq 0$.
\end{enumerate}
\end{proposition}

It is important to mention that the coefficients of the polynomial expressions $p_{\alpha, r}$ and $p_{\alpha, \beta}$ in Proposition \ref{coefficientes} are several evaluations of $r$ in $\sigma$'s and $\delta$'s depending on the coordinates of $\alpha$. A more explicit description of these relations can be found in \cite[Proposition 2.9 and Remark 2.10 (iv)]{Reyes2015}.

According to Definition \ref{def.skewpbwextensions}, if $A$ is a skew PBW extension over a ring $R$, then $A$ is a free left $R$-module. In this way, if $M_R$ is a right module, we can consider the set $M\langle X\rangle_A$ where the elements are of the form $m_1x^{\alpha_1} + \cdots + m_kx^{\alpha_k}, m_i \in M_R$ and $x^{\alpha_i} \in {\rm Mon}(A)$, for every $1 \le i \le k$. Considering Proposition \ref{coefficientes} and the explicit relations described in \cite[Proposition 2.9 and Remark 2.10 (iv)]{Reyes2015}, the set $M\langle X\rangle_A$ has $A$-module structure defined as follows: if $mx^{\alpha_i}:=mx_{1}^{\alpha_{i1}} \cdots x_{n}^{\alpha_{in}} \in M\langle X\rangle_A$ and $bx^{\beta_j}:=bx_{1}^{\beta_{j1}} \cdots x_{n}^{\beta_{jn}} \in A$, then we multiply these elements following the rule 
{\footnotesize \begin{equation*}\label{ecuacion2}
    \begin{split}
     \begin{aligned}
    mx_{1}^{\alpha_{i1}} \cdots x_{n}^{\alpha_{in}}bx_{1}^{\beta_{j1}} \cdots x_{n}^{\beta_{jn}} &= m\sigma^{\alpha_i}(b)x^{\alpha_i}x^{\beta_j} + mp_{\alpha_{i1}, \sigma_{i2}^{\alpha_{i2}}(\cdots(\sigma_{in}^{\alpha_{in}}(b)))}x_2^{\alpha_{i2}} \cdots x_n^{\alpha_{in}} x^{\beta_j}\\
    &\ \ \ +mx_1^{\alpha_{i1}}p_{\alpha_{i2}, \sigma_{i3}^{\alpha_{i3}}(\cdots(\sigma_{in}^{\alpha_{in}}(b)))}x_3^{\alpha_{i3}} \cdots x_n^{\alpha_{in}} x^{\beta_j}\\
   &\ \ \ + mx_1^{\alpha_{i1}}x_2^{\alpha_{i2}}p_{\alpha_{i3}, \sigma_{i4}^{\alpha_{i4}}(\cdots(\sigma_{in}^{\alpha_{in}}(b)))}x_4^{\alpha_{i4}} \cdots x_n^{\alpha_{in}} x^{\beta_j}\\
    &\ \ \ + \cdots + mx_1^{\alpha_{i1}}x_2^{\alpha_{i2}}\cdots x_{(n-2)}^{\alpha_{i(n-2)}}p_{\alpha_{i(n-1)}, \sigma_{in}^{\alpha_{in}}(b)}x_n^{\alpha_{in}}x^{\beta_j}\\
    &\ \ \ + mx_1^{\alpha_{i1}}x_2^{\alpha_{i2}}\cdots x_{(n-1)}^{\alpha_{i(n-1)}}p_{\alpha_{in}, b}x^{\beta_j}.
\end{aligned}  
    \end{split}
\end{equation*}}

This guarantees that $M\langle X\rangle_A$ is a right $A$-module and is called the {\it induced module} of $M_R$. If $\succeq$ is a total order defined on ${\rm Mon}(A)$ and $m \in M\langle X \rangle_A$ with $m \neq 0$, then we use expressions as $m=m_1x^{\alpha_1} + \cdots +m_kx^{\alpha_k}$, where $m_i\in M_R$, and $x^{\alpha_k}\succ \dotsb \succ x^{\alpha_1}$. With this notation, we define ${\rm
lm}(m):=x^{\alpha_k}$, the \textit{leading monomial} of $m$; ${\rm lc}(m):=m_k$, the \textit{leading coefficient} of $m$; ${\rm lt}(m):=m_kx^{\alpha_k}$, the \textit{leading term} of $m$. Note that $\deg(m):={\rm max}\{\deg(x^{\alpha_i})\}_{i=1}^k$. More details and properties about the induced module $M\langle X \rangle_A$ can be consulted in \cite{NinoRamirezReyes, Reyes2019}.

\section{Good polynomials and uniform dimension}\label{Goodpolynomials}
In this section, we study good polynomials, the essential modules and the uniform dimension of induced modules over skew PBW extensions.

\subsection{Good polynomials over skew PBW extensions}  

Since Leroy and Matczuk \cite{LeroyMatczuk2004} used the good polynomials to study the uniform dimension and the associated primes of induced modules over skew polynomial rings, we introduce the following definition thinking about skew PBW extensions.

\begin{definition}\label{goodpolynomial} 
Let $A$ be a skew PBW extension over $R$, $M_R$ be a right module, and $m = m_1x^{\alpha_1}+ \cdots +m_kx^{\alpha_k} \in M\langle X \rangle_A$ with leading coefficient $m_k \neq 0$. We say that  $m$ is a {\it good polynomial} if for every $r \in R$, ${\rm lm}(mr) = {\rm lm}(m)$, as long as $mr \neq 0$.
\end{definition}

Following Leroy and Matczuk, if $R$ is a ring, $\sigma$ is an automorphism of $R$, and $M_R$ is a right module, then $M_{\sigma}$ denotes the {\em $\sigma$-twisted module} defined on the same additive structure $M_{\sigma} = M$, where the action is defined by $m\cdot r := m\sigma(r)$, for all $r \in R$ \cite[p. 2747]{LeroyMatczuk2004}. With this in mind, we consider the following definition. If $\Sigma= \{\sigma_1, \ldots, \sigma_n \}$ is a finite set of automorphisms of $R$ and $\alpha =(\alpha_1, \ldots, \alpha_n)\in \mathbb{N}^n$, then $M_{\sigma^{\alpha}}$ denotes the {\em $\sigma^{\alpha}$-twisted module} where the action of $R$ over $M_{\sigma^{\alpha}}$ is given by  $m\cdot r := m\sigma^{\alpha}(r)=m\sigma_1^{\alpha_1}\circ \dotsb \circ \sigma_n^{\alpha_n}(r)$, for every $r \in R$. If $m \in M$, we denote by $\langle m \rangle_{\sigma^{\alpha}}$ the {\em $\sigma^{\alpha}$-twisted $R$-module generated by} $m$. Additionally, if $X$ is a subset of $R$ and $\alpha \in \mathbb{N}^n$, then $\sigma^{-\alpha}(X)$ denotes the {\em inverse image} of $X$ under $\sigma^{\alpha}$, that is, if $r \in \sigma^{-\alpha}(X)$, then $\sigma^{\alpha}(r) \in X$.

\begin{lemma}\label{PBWLemma3.4}
Let $A$ be a bijective skew PBW extension over $R$, $M_R$ be a right module, and $m = m_1x^{\alpha_1} + \cdots +m_kx^{\alpha_k} \in M\langle X \rangle_A$ with leading coefficient $m_k\neq 0 $. The following statements are equivalent:
\begin{itemize}
    \item [(1)]$m$ is a good polynomial.
    
    \item [(2)]For all $f \in mR_R$, ${\rm lm}(f) \succeq x^{\alpha_k}$.
    
    \item[(3)] For all $f \in mA_A$, ${\rm lm}(f) \succeq x^{\alpha_k}$.
    
    \item[(4)] For any $r \in R$, $m_k\sigma^{\alpha_k}(r) = 0$ if and only if $mr = 0$.
    
    \item[(5)] ${\rm ann}_R(m) = \sigma^{-\alpha_k}({\rm ann}_R(m_k))$.
    
    \item[(6)] ${\rm ann}_A(m) = \sigma^{-\alpha_k}({\rm ann}_R(m_k))A$.
    
    \item[(7)] $mA \cong \langle m_k\rangle_{\sigma^{\alpha_k}}A$ as $A$-modules.
\end{itemize}
\end{lemma}
\begin{proof} 
\begin{itemize}
    \item[] $(1)\Rightarrow (2)$ If $f \in mR_R$, then $f=mr$, for some $r \in R$. If we assume that ${\rm lm}(f) \prec x^{\alpha_k}$, then ${\rm lm}(mr) \prec {\rm lm}(m)$, which contradicts that $m$ is a good polynomial. Hence $x^{\alpha_k} \preceq {\rm lm}(f)$, for all $f \in mR_R$.  
    
    \item[]$(2)\Rightarrow (3)$ Suppose that $x^{\alpha_k} \preceq {\rm lm}(f)$, for all $f \in mR_R$. It is clear that ${\rm lm}(m)\preceq {\rm lm}(mx^{\alpha})$, for every $\alpha \in \mathbb{N}^n$. Thus, if $f=mg$, for some $g \in A$, then $x^{\alpha_k}= {\rm lm}(m) \prec {\rm lm}(mg)$, whence $x^{\alpha_k} \preceq {\rm lm}(f)$, for all $f \in mA_A$.
    
    \item[]$(3)\Rightarrow (4)$ Suppose that $m_k\sigma^{\alpha_k}(r)=0$. If $mr \neq 0$, then $mr$ is a non-zero element of $mA_A$ such that ${\rm lm}(mr) \prec {\rm lm}(m)=x^{\alpha_k}$, which is a contradiction. Thus, we have $mr=0$. For the other implication, if $mr=0$ then the leading coefficient of $mr$ is zero, that is, $m_k\sigma^{\alpha_k}(r)=0$ as desired.   
    
    \item[]$(4)\Rightarrow (5)$ If $r \in {\rm ann}_R(m)$, then $mr=0$. By statement $(4)$, $m_k\sigma^{\alpha_k}(r)=0$ whence $\sigma^{\alpha_k}(r) \in {\rm ann}(m_k)$. Hence, $r \in \sigma^{-\alpha_k}({\rm ann}(m_k))$, and so ${\rm ann}_R(m) \subseteq \sigma^{-\alpha_k}({\rm ann}_R(m_k))$. For the other inclusion if $r \in \sigma^{-\alpha_k}({\rm ann}_R(m_k))$, then $\sigma^{\alpha_k}(r) \in {\rm ann}(m_k)$ which implies that $m_k\sigma^{\alpha_k}(r)=0$, and thus $mr=0$ by hypothesis. This proves that $\sigma^{-\alpha_k}({\rm ann}_R(m_k)) \subseteq {\rm ann}_R(m)$.
    
    \item[]$(5)\Rightarrow (6)$ Assume that (5) holds, and let $g= b_1x^{\beta_1} + \cdots + b_tx^{\beta_t} \in {\rm ann}_A(m)$. If $mg=0$, then $m_k\sigma^{\alpha_k}(b_t)=0$, and so $b_t \in {\rm ann}_R(m)$ by hypothesis. Thus, $m\left (b_1x^{\beta_1} + \cdots + b_{t-1}x^{\beta_{t-1}} \right ) = 0$ whence $m_k\sigma^{\alpha_k}(b_{t-1})=0$, and so $b_{t-1} \in {\rm ann}_R(m)$. Continuing this process, we have $b_i \in \sigma^{-\alpha_k}({\rm ann}_R(m_k))$, for every $1 \le i \le t$, and so $g \in \sigma^{-\alpha_k}({\rm ann}_R(m_k))A$. This proves the inclusion ${\rm ann}_A(m) \subseteq \sigma^{-\alpha_k}({\rm ann}_R(m_k))A$. The other inclusion is clear.
    
    \item[]$(6)\Rightarrow (7)$ Assume that (6) holds. If $\phi$ is the $A$-module homomorphism of $A$ over $mA$ defined by $\phi(f)=mf$ for all $f \in A$, then it is clear that the kernel of $\phi$ is ${\rm ann}_A(m)$, whence $mA \cong A / {\rm ann}_A(m)$ by the first isomorphism theorem for $A$-modules. Additionally, if ${\rm ann}_A(m) = \sigma^{-\alpha_k}({\rm ann}_R(m_k))A$ by assumption, then we get the isomorphism $mA \cong A / \sigma^{-\alpha_k}({\rm ann}_R(m_k))A$. Let $\varphi$ be the map of $R / \sigma^{-\alpha_k}({\rm ann}_R(m_k)) \times A$ over $A / \sigma^{-\alpha_k}({\rm ann}_R(m_k))A$ defined by $\varphi(\overline{r},f):=\overline{rf}$. It is not difficult to see that $\varphi$ is bilinear and due to the universal property of the tensorial product, there exists $\overline{\varphi}$ of $R / \sigma^{-\alpha_k}({\rm ann}_R(m_k)) \otimes_R A$ over $A / \sigma^{-\alpha_k}({\rm ann}_R(m_k))A$ give by $\overline{\varphi}(\overline{r}\otimes f):=\overline{rf}$ with inverse $\overline{\varphi}^{-1}$ of $A / \sigma^{-\alpha_k}({\rm ann}_R(m_k))A$ over $R / \sigma^{-\alpha_k}({\rm ann}_R(m_k)) \otimes_R A$ defined by $\overline{\varphi}^{-1}(\overline{f}):= \overline{1} \otimes f$. Thus, it follows that $A / \sigma^{-\alpha_k}({\rm ann}_R(m_k))A \cong  R / \sigma^{-\alpha_k}({\rm ann}_R(m_k)) \otimes_R A$. Since $R / \sigma^{-\alpha_k}({\rm ann}_R(m_k))$ is isomorphic to the $R$-module $\langle m_k \rangle_{\sigma^{\alpha_k}}$, we have $mA \cong \langle m_k\rangle_{\sigma^{\alpha_k}}A$, where the isomorphism $\phi$ is defined as $\phi(mg) := m_k \sigma^{\alpha_k}(b_1)x^{\beta_1} + \cdots + m_k \sigma^{\alpha_k}(b_t)x^{\beta_t}$, for every $g=b_1x^{\beta_1} + \cdots + b_tx^{\beta_t} \in A$. 
    
    \item[]$(7)\Rightarrow (1)$ Let $\phi$ be the isomorphism defined above. If $m$ is a polynomial such that ${\rm lm}(mr)\prec {\rm lm}(m)$, for some $r \in R$, then $m_k\sigma^{\alpha_k}(r)=0$. Thus, $\phi(mr)=m_k\sigma^{\alpha_k}(r)=0$, and since $\phi$ is injective, this implies that $mr=0$. Hence, $m$ is a good polynomial.
\end{itemize}
\end{proof}


Shock proved that given any non-zero polynomial $f(x) \in R[x]$, there exists $r \in R$ such that $f(x)r$ is good \cite[Proposition 2.2]{Shock1972}. Lemma \ref{PBWCorollary3.5} generalize Shock's result and characterizes the right annihilators of good polynomials over the induced module $M\langle X \rangle_A$. This lemma extends \cite[Corollary 3.5]{LeroyMatczuk2004}.

\begin{lemma} \label{PBWCorollary3.5}
Let $A$ be a bijective skew PBW extension over $R$ and $M_R$ be a right module. If  $m = m_1x^{\alpha_1} + \cdots + m_kx^{\alpha_k} \in M\langle X \rangle_A$ with leading coefficient $m_k \neq 0$, then:
\begin{itemize}
    \item [\rm (1)] There exists $r \in R$ such that $mr$ is good polynomial. 
    
    \item [\rm (2)] If $m$ is good polynomial, we have ${\rm ann}_A(m) = {\rm ann}_R(m)A$.
\end{itemize}
\end{lemma}
\begin{proof}
\begin{itemize}
    \item [\rm (1)] Assume that the result is false and suppose that $m \in M\langle X \rangle_A$ is a counterexample of minimal leading monomial, that is, $mr$ is not a good polynomial, for every $r \in R$. In particular, if $r=1$, then $m$ is not a good polynomial, and thus there exists $r \in R$ such that ${\rm lm}(mr) \prec {\rm lm}(m)$. If $mr \neq 0$, then by the minimality of ${\rm lm}(m)$, there exists $c \in R$ with $mrc$ a good polynomial. However, this contradicts the fact that $mr$ is not a good polynomial for all $r \in R$. 

    \item [\rm (2)] If $m$ is good polynomial, by using equivalences $(5)$ and $(6)$ of Lemma \ref{PBWLemma3.4}, we have ${\rm ann}_A(m) = \sigma^{-\alpha_k}({\rm ann}_R(m_k))A ={\rm ann}_R(m)A$.
\end{itemize}
\end{proof}


A submodule $N_R$ of $M_R$ is called {\it essential} if $mR \cap N \neq 0$, for any $0 \neq m \in M_R$. The set of all elements $m \in M_R$ such that ${\rm ann}_R(m)$ is an essential ideal of $R$ is said to be the {\em singular submodule} of $M_R$ and is denoted by $Z(M_R)$; $M_R$ is called {\it nonsingular} if $Z(M_R) = 0$ \cite[Definition 3.26]{Lam1998}. Leroy and Matczuk \cite{LeroyMatczuk2004} proved that there exist good polynomials of any degree in a submodule of a nonsingular module. The following proposition extends \cite[Proposition 3.6]{LeroyMatczuk2004}.

\begin{proposition}\label{PBWProposition3.6}
Let $A$ be a bijective skew PBW extension over $R$, $M_R$ be a right module, and $m = m_1x^{\alpha_1} + \cdots + m_kx^{\alpha_k}$ be a good polynomial of $M\langle X \rangle_A$ with leading coefficient $m_k \neq 0$. If the submodule $m_kR$ of $M_R$ is nonsingular, then:
\begin{itemize}
    \item[\rm (1)] $m_kA$ contains a good polynomial $f$ with leading monomial $x_i$ for all $i$.
    \item[\rm (2)] For any $x^{\alpha_l} \succeq x^{\alpha_k} $, there exists a good polynomial  $f \in mA$ with leading monomial $x^{\alpha_l}$.
\end{itemize}
\end{proposition}
\begin{proof}
\begin{itemize}
    \item[(1)] Let $g \in A$ with leading term $m_kx_i$ for some $1 \le i \le n$. Since $m_kR$ is nonsingular, there exists $0\neq b \in R$ such that $\sigma_i(b)R \bigcap {\rm ann}_R(m_k)=0$. Note that the leading coefficient of the polynomial $gb$ is $m_k\sigma_i(b)$. Now, consider an element $r \in R$ such that $\sigma_i(r) \in {\rm ann}_R(m_k\sigma_i(b))$. In this way, $\sigma_i(b)\sigma_i(r) \in \sigma_i(b)R \bigcap {\rm ann}_R(m_k)=0$, and so $\sigma_i(br)=0$. By the injectivity of $\sigma_i$, if $br=0$, then $gbr=0$. This proves that $f=gb$ is a good polynomial with leading monomial $x_i$. 
   
    \item[(2)] If $m$ is a good polynomial, then there exists an isomorphism of $A$-modules $\phi$ of $mA$ over $\langle m_k\rangle_{\sigma^{\alpha_k}}A$ such that $\phi(mg) := m_k\sigma^{\alpha_k}(b_1)x^{\beta_1} + \cdots +m_k \sigma^{\alpha_k}(b_j)x^{\beta_j} $, for every $g= b_1x^{\beta_1} + \cdots + b_jx^{\beta_j} \in A$, by Lemma \ref{PBWLemma3.4} (7). Furthermore, we get $\sigma^{-\alpha_k}({\rm ann}_R(m_k))={\rm ann}_R(m)$, by Lemma \ref{PBWLemma3.4} (5), which implies that the leading monomial of $\phi(mg)$ is $x^{\beta_j}$ if and only if the leading monomial of $mg$ is $x^{\alpha_k + \beta_j}$. Therefore, $mg \in mA$ is a good polynomial if and only if $m_k{\sigma^{\alpha_k}}(g) \in \langle m_k \rangle_{\sigma^{\alpha_k}}A$ is a good polynomial by Lemma \ref{PBWLemma3.4} (2).

    It is not difficult to see that, if $m_kR$ is nonsingular, then $\langle m_k\rangle_{\sigma^{\alpha_k}}$ also is. By part (1), there exists $g'=\phi(mg) \in \langle m_k\rangle_{\sigma^{\alpha_k}}A$ a good polynomial such that the leading monomial of $g'$ is $x_1$. Thus, $mg'$ is a good polynomial of $mA$ with leading monomial $x^{\beta}$ with $\beta \in \mathbb{N}^n$, where $\beta_1=\alpha_{k1}+1$ and $\beta_i=\alpha_{ki}$, for all $2 \le i \le n$. Since the leading coefficient of $mg'$ belongs to $\langle m_kR \rangle_{\sigma^{\alpha_k}}$, its leading coefficient satisfies the hypothesis of the theorem. Thus, there exists $mg'' \in mA$ with leading monomial $x^{\beta}$ where $\beta_1=\alpha_{k1}+2$ and $\beta_i=\alpha_{ki}$, for every $2 \le i \le n$. Following this argument, in at most $|\alpha_{l1} - \alpha_{k1}|$ steps, we find a good polynomial $mg_1$ with leading monomial $x^{\beta}$ where $\beta_1=\alpha_{l1}$ and $\beta_i=\alpha_{ki}$, for all $2 \le i \le n$. The idea is to continue with $x_2$. By part (1), there exists $g_1' \in \langle m_k\rangle_{\sigma^{\alpha_k}}A$ a good polynomial where the leading monomial of $g_1'$ is $x_2$. Thus, $mg_1'$ is a good polynomial of $mA$ with leading monomial $x^{\beta}$ with $\beta \in \mathbb{N}^n$, where $\beta_1=\alpha_{l1}$, $\beta_{2}=\alpha_{k2}+1$ and $\beta_i=\alpha_{ki}$, for all $3 \le i \le n$. Repeating the process, in at most $|\alpha_{l2} - \alpha_{k2}|$ steps, we find a good polynomial $mg_2$ with leading monomial $x^{\beta}$ where $\beta_i=\alpha_{li}$ for $i=1,2$ and $\beta_i=\alpha_{ki}$, for all $3 \le i \le n$. In this way, in at most $n\cdot {\rm max}\{|\alpha_{li} - \alpha_{ki}| \}$ steps, we have a good polynomial $f$ with ${\rm lm}(f)=x^{\alpha_l}$.
\end{itemize}
\end{proof}

Let $R$ be a ring, $\Sigma:=\{\sigma_1, \ldots, \sigma_n \}$ be a finite set of endomorphisms of $R$ and $\Delta:=\{\delta_1, \ldots, \delta_n \}$ be a finite set of $\Sigma$-derivations of $R$. If $I$ is a two-sided ideal of $R$, then $I$ is called $\Sigma$-{\em invariant} if $\sigma_i(I)\subseteq I$,  for every $1 \le i \le n$; $I$ is a $\Delta$-{\em invariant} ideal if $\delta_i(I)\subseteq I$, for all $1 \le i \le n$; if $I$ is both $\Sigma$ and $\Delta$-invariant, we say that $I$ is {\em $(\Sigma,\Delta)$-invariant} \cite[Definition 2.1]{LezamaAcostaReyes2015}. These ideals have been widely studied in the literature \cite{LezamaAcostaReyes2015, LouzariReyes2019, NinoReyes2019, Reyes2014}.

Leroy and Matczuk defined some invariant ideals associated with an ideal $I$ of $R$, an endomorphism $\sigma$ of $R$, and a $\sigma$-derivation $\delta$ of $R$ (cf. \cite[p. 2746]{LeroyMatczuk2004}). Following their ideas, we consider some ideals with the aim of studying properties of induced modules over skew PBW extensions.

\begin{definition} Let $R$ be a ring, $\Sigma:=\{\sigma_1, \ldots, \sigma_n \}$ be a finite set of endomorphisms of $R$ and $\Delta:=\{\delta_1, \ldots, \delta_n \}$ be a finite set of  $\Sigma$-derivations of $R$. If $I$ is a two-sided ideal of $R$, then
\begin{align*}
    I_{\Sigma}&:= \{a \in I \ | \ \sigma^{\alpha}(a) \in I,\ \text{for all}\ \alpha \in \mathbb{N}^n \},\\
    I_{\Delta}&:= \{a \in I \ | \ \delta^{\beta}(a) \in I,\ \text{for all}\ \beta \in \mathbb{N}^n \},\\
    I_{\Sigma,\Delta}&:= \{a \in I \ | \ \sigma^{\alpha}\delta^{\beta}(a) \in I,\ \text{for all}\ \alpha,\beta \in \mathbb{N}^n\}.
\end{align*}    
\end{definition}

\begin{remark}
Notice that $I$ is $\Sigma$-invariant if and only if $I=I_{\Sigma}$; $I$ is $\Delta$-invariant if and only if $I=I_{\Delta}$, and $I$ is a $(\Sigma, \Delta)$-invariant ideal if and only if $I=I_{\Sigma,\Delta}$.   
\end{remark}

The following lemma extends \cite[Lemma 3.1]{LeroyMatczuk2004}. 

\begin{lemma}\label{PBWLemma3.1} Let $R$ be a ring, $\Sigma:=\{\sigma_1, \ldots, \sigma_n \}$ be a finite set of endomorphisms of $R$, and $\Delta:=\{\delta_1, \ldots, \delta_n \}$ be a finite set of $\Sigma$-derivations of $R$.
\begin{itemize} 
    \item[{\rm (1)}] $ I_{\Sigma, \Delta} \subseteq I_{\Sigma} \cap I_{\Delta}$.
    
    \item[{\rm (2)}] If either $I_{\Sigma}$ is $\Delta$-invariant or $I_{\Delta}$ is $\Sigma$-invariant, then $I_{\Sigma, \Delta}=I_{\Sigma}$.
    
    \item[{\rm (3)}] If $N_R$ is a right module with $I={\rm ann}_R(N)$ an ideal $\Delta$-invariant, then:
    \begin{itemize}
        \item[{\rm (a)}] ${\rm ann}_R(N\left \langle X \right \rangle) = I_{\Sigma, \Delta}$.
        \item[{\rm (b)}] If ${\rm ann}_A(N\left \langle X \right \rangle) = JA$ for some ideal $J$, then $J = I_{\Sigma, \Delta} = {\rm ann}_R(N\left \langle X \right \rangle)$. 
    \end{itemize}
\end{itemize}
\end{lemma}
\begin{proof}
\begin{itemize}
\item[\rm (1)] If $a \in I_{\Sigma,\Delta}$, then $\sigma^{\alpha}\delta^{\beta}(a) \in I$ for every $\alpha, \beta \in \mathbb{N}^n$. Thus, $\sigma^{\alpha}(a) \in I$ which implies that $a \in I_{\Sigma}$. Similarly, if $\sigma^{\alpha}\delta^{\beta}(a) \in I$ for all $\alpha, \beta \in \mathbb{N}^n$, then $\delta^{\beta}(a) \in I$ for all $\beta \in \mathbb{N}^n$, which implies that $a \in I_{\Delta}$, and so $a \in I_{\Sigma}\cap I_{\Delta}$. 

\item[\rm (2)] Suppose that $I_{\Sigma}$ is a $\Delta$-invariant ideal. If $a \in I_{\Sigma}$ and $I_{\Sigma}$ is $\Delta$-invariant, then $\delta^{\beta}(a)\in I_{\Sigma}$, and hence $\sigma^{\alpha}\delta^{\beta}(a) \in I$, for all $\alpha,\beta \in \mathbb{N}^n$. Thus, $I_{\Sigma}\subseteq I_{\Sigma,\Delta}$. If $I_{\Delta}$ is $\Sigma$-invariant, then the argument is similar.

\item[\rm (3)] Let $N_R$ be a right module with $I={\rm ann}_R(N)$ an ideal $\Delta$-invariant.
    \begin{itemize}
        \item[\rm (a)] Let $n= n_1x^{\alpha_1} + \cdots + n_kx^{\alpha_k} \in N\langle X \rangle_A$. For every $r\in R$, the coefficients of $nr$ are products of $n_i$ with elements obtained evaluating $\sigma$'s and $\delta$'s in the element $r$ \cite[Remark 2.10]{Reyes2015}. If $r \in I_{\Sigma,\Delta}$, then $nr=0$, and so $I_{\Sigma,\Delta} \subseteq {\rm ann}_R(N\langle X \rangle)$. For the other inclusion, if $r \in {\rm ann}_R(N\langle X \rangle)$, then $nx^{\alpha}r = 0$, for any $\alpha \in \mathbb{N}^n$ and every $n \in N$. Notice that the leading coefficient of $nx^{\alpha}r$ is zero, that is, $n\sigma^{\alpha}(r)=0$, for all $\alpha$, which implies that $r \in I_{\Sigma}$. Since $I$ is a $\Delta$-invariant ideal, we have $r \in I_{\Sigma}=I_{\Sigma,\Delta}$ and so ${\rm ann}_R(N\langle X \rangle) \subseteq I_{\Sigma,\Delta}$. 

        \item[\rm (b)] By item (a), we obtain $I_{\Sigma,\Delta} \subseteq J$. For the other inclusion, if we have the equality ${\rm ann}_A(N \langle X \rangle) = JA$, for some ideal $J$, then $nx^{\alpha}rx^{\beta} = 0$, for every $\alpha,\beta \in \mathbb{N}^n$, $r \in J$, and $n \in N$. In particular, if $\beta = 0$, then $nx^{\alpha}r = 0$ and $n\sigma^{\alpha}(r)=0$, for any $\alpha$, which implies that $r \in I_{\Sigma}$. Since $I$ is $\Delta$-invariant, we get $r \in I_{\Sigma}=I_{\Sigma,\Delta}$ which shows that the equality $J = I_{\Sigma,\Delta} = {\rm ann}_R(N\left \langle X \right \rangle)$ follows.
\end{itemize}
\end{itemize}
\end{proof}

The following lemma characterizes annihilators of generated submodules by good polynomials and extends \cite[Lemma 3.7]{LeroyMatczuk2004}.

\begin{lemma}\label{PBWLemma3.7} Let $A$ be a bijective skew PBW extension over $R$, $M_R$ be a right module, and $m= m_1x^{\alpha_1} + \cdots + m_kx^{\alpha_k} \in M \langle X \rangle_A$ be a good polynomial with leading coefficient $m_k \neq 0$. If $I:=\sigma^{-\alpha_k}({\rm ann}_R(m_kR))$, then:
\begin{itemize}
    \item[\rm (1)]${\rm ann}_R(mA) \subseteq {\rm ann}_A(mA) \subseteq {\rm ann}_A(mR)$.
    
    \item[\rm (2)] ${\rm ann}_R(mR) = I$ and ${\rm ann}_A(mR) = IA$.
    
    \item[\rm (3)]${\rm ann}_R(mA) = I_{\Sigma,\Delta}$.
    
   \item[\rm (4)]If $I_{\Sigma}$ is $\Delta$-invariant, then ${\rm ann}_A(mA) = I_{\Sigma}A$.
\end{itemize}
\end{lemma}
\begin{proof}
\begin{itemize}
    \item[\rm (1)] Since $R\subseteq A$, it follows that ${\rm ann}_R(mA) \subseteq {\rm ann}_A(mA)$. On the other hand if $f \in {\rm ann}_A(mA)$, then $gAf=0$ whence $gRf=0$. This proves that ${\rm ann}_A(mA)\subseteq {\rm ann}_A(mR)$.
    
    \item[\rm (2)] Let us show that ${\rm ann}_R(mR)=I$. If $r \in I$, then $m_kR\sigma^{\alpha_k}(r)=0$, and since $m$ is a good polynomial, we have $mRr=0$ and thus $r \in {\rm ann}_R(mR)$. Now, if $r \in {\rm ann}_R(mR)$, then $m_kR\sigma^{\alpha_k}(r)=0$ and so $r \in \sigma^{-\alpha_k}({\rm ann}_R(m_kR))=I$. Let us prove that ${\rm ann}_A(mR)=IA$. If $f= a_1x^{\beta_1} + \cdots + a_tx^{\beta_t} \in {\rm ann}_A(mR)$, then $mRf=0$ whence $m_kR\sigma^{\alpha_k}(a_t)=0$. Since $m$ is good, it follows that $a_t \in {\rm ann}_R(mR)=I$. If $mRa_t=0$, then $mR(a_1x^{\beta_1} + \cdots + a_{t-1}x^{\beta_{t-1}}) = 0$ which implies that $m_kR\sigma^{\alpha_k}(a_{t-1})=0$. If $m$ is good, then $mRa_{t-1}=0$, and thus $a_{t-1} \in {\rm ann}_R(mR)=I$. Continuing this argument, we have that $a_i \in {\rm ann}_R(mR)=I$, for every $1 \leq i \leq t$, and so $f \in {\rm ann}_R(mR)A=IA$. Hence, ${\rm ann}_A(mR) \subseteq IA$. The reverse inclusion is clear.
    
    \item[\rm (3)] We have the $R$-module isomorphism $(mR)\langle X \rangle = mR \otimes_R A\cong mA$. By item (2), we obtain that ${\rm ann}_R(mR)=I$ and by Lemma \ref{PBWLemma3.1} (3)(a), it follows that ${\rm ann}_R((mR)\langle X \rangle)={\rm ann}_R(mA)=I_{\Sigma, \Delta}$.
    
    \item[\rm (4)] Suppose that $I_{\Sigma}$ is $\Delta$-invariant. By Lemma \ref{PBWLemma3.1} (2), we get $I_{\Sigma} = I_{\Sigma,\Delta}$. By items $(1)$ and $(3)$, we obtain $I_{\Sigma}A = {\rm ann}_R(mA)A \subseteq {\rm ann}_A(mA)$. To see the other inclusion if $f = a_1x^{\beta_1} + \cdots + a_tx^{\beta_t} \in  {\rm ann}_A(mA)$, then $mRx^{\beta}f = 0$ whence $\sigma^{\beta}(a_t) \in \sigma^{-\alpha_k}({\rm ann}_R(m_kR)) = I$, for every $\beta \in \mathbb{N}^n$, and so $a_t \in I_{\Sigma}$. Since $m$ is good, we get $mRx^{\beta}a_t = 0$ for all $\beta \in \mathbb{N}^n$ whence $mAa_t = 0$. Then $mA(f - a_tx^{\beta_t}) = 0$ and ${\rm lm}(f - a_tx^{\beta_t}) \prec {\rm lm}(f)$.  An inductive argument proves that ${\rm ann}_A(mA) \subseteq I_{\Sigma}A$ and so the equality ${\rm ann}_A(mA) = I_{\Sigma}A$ holds.
\end{itemize}
\end{proof}

\subsection{Uniform dimension over skew PBW extensions} 

In this section, we study the essential modules and the uniform dimension of induced modules over skew PBW extensions. The following theorem extends \cite[Lemma 4.1]{LeroyMatczuk2004}.

\begin{lemma}\label{PBWLemma4.1}
Let $A$ be a bijective skew PBW extension over $R$ and $M_R$ be a right module. If $N_R$ is a submodule $M_R$, then:
\begin{itemize}
    \item[\rm (1)] $N\langle X \rangle_R$ is essential in  $M\langle X \rangle_R$ if and only if $N\langle X \rangle_A$ is essential in $M\langle X \rangle_A$. \item[(2)]${Z}(M_R)=0$ if and only if ${Z}(M\langle X \rangle_R)=0$.
    
    \item[\rm (3)] If ${Z}(M_R)=0$, then ${Z}(M \langle X \rangle_A)=0$.
\end{itemize}
\end{lemma}
\begin{proof} 
\begin{itemize}
    \item[\rm (1)] If $N\langle X \rangle_R$ is essential, then $0 \neq mR \cap N\langle X \rangle_R \subseteq mA \cap N\langle X \rangle_A$ for every $m\in M\langle X \rangle_A$, and so $N\langle X \rangle_A$ is an essential submodule of $M\langle X \rangle_A$. Suppose that $N\langle X \rangle_A$ is an essential submodule of $M\langle X \rangle_A$ and consider the $A$-module homomorphism of $M\langle X \rangle_A$ over $(M/N)\langle X \rangle_A$ defined by $\phi(m_1x^{\alpha_1} + \cdots + m_kx^{\alpha_k}):= \overline{m_1}x^{\alpha_1} + \cdots + \overline{m_k}x^{\alpha_k}$. It is not hard to see that $\phi$ induces an isomorphism $\overline{\phi}$ between the $A$-modules $M\langle X \rangle_A / N\langle X \rangle_A$ and $(M/N)\langle X \rangle_A$. Let $\overline{f}\in (M/N)\langle X \rangle_A$ be the image of $f$ by $\overline{\phi}$, for some $f \in M\langle X \rangle_A \setminus N\langle X \rangle_A$. By Lemma \ref{PBWCorollary3.5} (1), there exists $r \in R$ such that $\overline{f}r$ is a good polynomial and since $N\langle X \rangle_A$ is an essential submodule of $M\langle X \rangle_A$, we have $frA  \cap N\langle X \rangle_A\neq 0$, that is, there exists $g= b_1x^{\beta_1} + \cdots+ b_tx^{\beta_t} \in A$ such that $0\neq frg\in N\langle X \rangle_A$. By Lemma \ref{PBWCorollary3.5} (2), if $0\neq frg\in N\langle X \rangle_A$, then $g \in {\rm ann}_A(\overline{f}r)= {\rm ann}_R(\overline{f}r)A$ whence $b_i \in {\rm ann}_R(\overline{f}r)$, for every $1 \leq i \leq t$. Thus, $\overline{f}rb_i=0$, and so $frb_i \in N\langle X \rangle_R$. Additionally if $frg \neq 0$, then $0\neq frb_i \in N \langle X \rangle_R$, for some $1\leq i \leq t$. This shows that $ fR \cap N\langle X \rangle_R \neq 0$, and hence $N\langle X \rangle_R$ is an essential submodule of $M\langle X \rangle_R$.
    
    \item[\rm (2)] Since $Z(M_R) = Z(M \langle X \rangle_R) \cap M_R$ if $Z(M\langle X \rangle_R)=0$, then $Z(M_R)=0$. Assume that $Z(M\langle X \rangle_R) \neq 0$ and let $f=m_1x^{\alpha_1} + \cdots + m_kx^{\alpha_k} \in Z(M\langle X \rangle_R)$ with leading coefficient $m_k\neq 0$ and minimal monomial leading, that is, ${\rm lm}(f) \preceq {\rm lm}(g)$, for all $g \in fR$. By Lemma \ref{PBWLemma3.4} (2) and (5), $f$ is a good polynomial and ${\rm ann}_R(f)=\sigma^{-\alpha_k}({\rm ann}_R(m_k))$ is an essential ideal of $R$ which implies that  $0 \neq m_k \in Z(M\langle X \rangle_R) \cap M_R=Z(M_R)$.

    \item[(3)] Assume that $Z(M_R) = 0$ and $Z(M\langle X \rangle_A) \neq 0$. Let $f=m_1x^{\alpha_1} + \cdots + m_kx^{\alpha_k}$ be an element of $Z(M\langle X \rangle_A)$ with leading coefficient $m_k\neq 0$ and minimal monomial leading, that is, ${\rm lm}(f) \preceq {\rm lm}(g)$, for all $g \in fA$. By Lemma \ref{PBWLemma3.4} (3), $f$ is a good polynomial and by Lemma \ref{PBWCorollary3.5} (2), ${\rm ann}_A(f) = \sigma^{-\alpha_k}(I)A$ where $I = {\rm ann}_R(a_k)$. Since $\sigma^{-\alpha_k}(I)A \cap m'A \neq 0$, for all $m'\in M\langle X \rangle_A$, it follows that $\sigma^{-\alpha_k}(I) \cap m'M_R \neq 0$, and so $\sigma^{-\alpha_k}(I)$ is an essential ideal of $R$. Therefore, $0 \neq a_k \in Z(M_R)$ which is a contradiction. This proves that $Z(M\langle X \rangle_A) = 0$ as desired. 
\end{itemize}
\end{proof}



If $R$ is a ring, $\sigma$ is an automorphism of $R$, and $\delta$ is a $\sigma$-derivation of $R$, then an ideal $I$ of $R$ is said to be {\em $(\sigma,\sigma^{-1},\delta)$-stable} if $\sigma(I)=I$ and $\delta(I)\subseteq I$ \cite[p. 6]{GoodearlLetzter1994}. Annin \cite{Annin2002} investigated properties of induced modules over skew polynomial rings under the assumption that for any $m \in M_R$, $I={\rm ann}_R(m)$ is a $(\sigma,\sigma^{-1},\delta)$-stable ideal. Leroy and Matczuk \cite{LeroyMatczuk2004} studied these same modules using the following condition: $M_R$ satisfies the {\em weak $(\sigma, \delta)$-compatibility condition} if every non-zero submodule $N_R$ of $M_R$ contains an element $m\neq 0$ such that ${\rm ann}_R(m)$ is $(\sigma, \sigma^{-1}, \delta)$-stable \cite[Definition 4.2]{LeroyMatczuk2004}. Thinking about skew PBW extensions, we consider the following definition: if $R$ is a ring, $\Sigma= \{\sigma_1,\ldots,\sigma_n\}$ is a finite family of endomorphisms of $R$ and $\Delta= \{\delta_1,\ldots,\delta_n\}$ is a finite family of $\Sigma$-derivations of $R$, then an ideal $I$ of $R$ is called {\em $(\Sigma,\Sigma^{-1},\Delta)$-stable} if $\sigma_i(I)=I$ and $\delta_i(I)\subseteq I$, for every $1 \le i \le n$. This definition allows us to extend the compatibility condition as follows.

\begin{definition} Let $R$ be a ring, $\Sigma$ be a finite family of endomorphisms of $R$, and $\Delta$ a finite family of $\Sigma$-derivations of $R$.  We say that a right module $M_R$ satisfies the {\it weak $(\Sigma, \Delta)$-compatibility condition}, if every submodule $N_R$ of $M_R$ contains an element $a \neq 0$ such that $I:={\rm ann}_R(a)$ is $(\Sigma, \Sigma^{-1}, \Delta)$-stable.
\end{definition}

The following lemma is analogous to \cite[Lemma 4.3]{LeroyMatczuk2004}.

\begin{lemma}\label{PBWLemma4.3}
Let $R$ be a ring, $\Gamma=\{\gamma_{i},\ldots,\gamma_{n}\}$ be a family of automorphisms of $R$, and $\Lambda=\{\lambda_{i},\ldots,\lambda_{n}\}$ be a family of $\Gamma$-derivations of $R$. If $A$ is a bijective skew PBW extension over $R$ and $M_R$ is a right weak $(\Gamma, \Lambda)$-compatible module, then:
\begin{itemize}
    \item [\rm (1)] For any good  polynomial  $m \in M\langle X \rangle_A$, there exists $r \in R$ such that $mr$ is good and the annihilator of its leading coefficient is $(\Gamma,\Gamma^{-1},\Lambda)$-stable.
    
    \item[\rm (2)]  Suppose that for the family of endomorphisms $\Sigma=\{\sigma_1, \ldots, \sigma_n\}$ we have $\sigma_i\gamma_j = \gamma_j\sigma_i$ for $1 \leq i,j \leq n$. If there exist some central invertible elements $q_{ij} \in R$ such that $\sigma_i\lambda_j = q_{ij}\lambda_j\sigma_i$ and $\gamma_i$ and $\lambda_i$ can be extended to $A$, then then module $M\langle X \rangle_A$ satisfies the weak $(\Gamma, \Lambda)$-compatibility condition.
\end{itemize}
\end{lemma}
\begin{proof}
\begin{itemize}
\item [\rm (1)]  Let $m= m_1x^{\alpha_1} + \cdots + m_kx^{\alpha_k}$ be a non-zero good polynomial with leading coefficient $m_k\neq 0$. If $M_R$ satisfies the weak $(\Gamma, \Lambda)$-compatibility condition, then there exists $r \in
R$ such that $m_kr\neq 0$ and $I = {\rm ann}_R(m_kr)$ is $(\Gamma,\Gamma^{-1},\Lambda)$-stable. By Lemma \ref{PBWLemma3.4}, $m\sigma^{-\alpha_k}(r)$ is a good polynomial with leading coefficient $m_kr$ and $I$ is the annihilator of its leading
coefficient.

\item[\rm (2)] Let $N\langle X \rangle$ be a submodule of $M\langle X \rangle_A$ and $m=m_1x^{\alpha_1} + \cdots + m_kx^{\alpha_k} \in N\langle X \rangle$ be a polynomial of minimal leading monomial in $mA$, that is, $x^{\alpha_k}\prec {\rm lm}(f)$ for every $f \in mA$. By Lemma \ref{PBWLemma3.4} and item (1), $m$ is a good polynomial, and so $m\sigma^{-\alpha_k}(r)$ is also a good polynomial for some $r \in R$ with $I = {\rm ann}_R(m_kr)$ $(\Gamma,\Gamma^{-1},\Lambda)$-stable. Additionally, ${\rm ann}_A(m) = \sigma^{-\alpha_k} (I)A$ by Lemma \ref{PBWLemma3.4}, and since $\sigma_i\gamma_j = \gamma_j\sigma_i$, $\sigma_i\lambda_j = q_{ij}\lambda_j\sigma_i$ for every $1 \leq i,j \leq n$ and some central invertible element $q_{ij}\in R$, and also $\lambda_i$ can be extended to $A$, it follows that ${\rm ann}_A(m)=\sigma^{-\alpha_k}(I)A$ is $(\Gamma, \Gamma^{-1} \Lambda)$-stable. 
\end{itemize}
\end{proof}


Leroy and Matczuk showed that the induced modules over skew polynomial rings have good polynomials of any degree \cite[Example 3.3]{LeroyMatczuk2004}. This fact motivated the following definition: a submodule $B_S$ of $\widehat{M}_S$ is called {\it good}, if for any good polynomial $g \in B_S$ and any $n \geq {\rm deg(g)}$, there exists a good polynomial of degree $n$ in $gS$ \cite[Definition 4.4]{LeroyMatczuk2004}. Following this idea and with the purpose of studying properties of induced modules over skew PBW extensions, we introduce the following definition.

\begin{definition}
Let $A$ be a skew PBW extension over $R$ and $M_R$ be a right module. A submodule $N \langle X \rangle_A$ of $M\langle X \rangle _A$ is called {\it good}, if for any good polynomial $ m=m_1x^{\alpha_1} + \cdots +m_kx^{\alpha_k} \in N\langle X \rangle_A$ and any monomial $x^{\beta}$ with $\beta \in \mathbb{N}^n$, there exists a good polynomial $f \in mA$ such that ${\rm lm}(f) = x^{\beta} \succeq x^{\alpha_k}$.
\end{definition}

The following lemma generalizes \cite[Lemma 4.5]{LeroyMatczuk2004}.

\begin{lemma}\label{PBWLemma4.5} Let $A$ be a bijective skew PBW extension over $R$ and $M_R$ be a right module. If one of the following conditions is satisfied
\begin{itemize}
    \item [\rm (1)] $M_R$ is nonsingular.
    \item[\rm (2)] $M_R = R_R$ and for any non-zero element $r \in R$, there is a good polynomial $m\in rA$ with leading monomial $x_i$ for all $i$.
    \item[\rm (3)] $A$ is a skew PBW extension of endomorphism type.
    \item[\rm (4)] $M_R$ satisfies the weak $(\Sigma, \Delta)$-compatibility condition.
\end{itemize}
then ${M}\langle X \rangle_A$ is a good module.
\end{lemma}
\begin{proof}
\begin{itemize}
    \item[\rm (1)] This statement is a direct consequence of Proposition \ref{PBWProposition3.6}. 
    
    \item[\rm (2)] If $f=a_1x^{\alpha_1} + \cdots + a_kx^{\alpha_k}$ is a good polynomial of $A$ with leading monomial $x^{\alpha_k}$, then there exists an $A$-module isomorphism $\phi$ of $fA$ over $\langle a_k\rangle_{\sigma^{\alpha_k}}A$ such that $\phi(fg) := a_k\sigma^{\alpha_k}(b_1)x^{\beta_1} + \cdots +a_k \sigma^{\alpha_k}(b_j)x^{\beta_j} $, for every $g \in A$ with leading monomial $x^{\beta_j}$ for any $\beta_j \in \mathbb{N}^{n}$, by Lemma \ref{PBWLemma3.4} (7). Furthermore, we have $\sigma^{-\alpha_k}({\rm ann}_R(a_k))={\rm ann}_R(f)$, by Lemma \ref{PBWLemma3.4} (5), which implies that the leading monomial of $\phi(fg)$ is $x^{\beta_j}$ if and only if the leading monomial of $fg$ is $x^{\alpha_k + \beta_j}$. Therefore, $fg \in fA$ is a good polynomial if and only if $a_k{\sigma^{\alpha_k}}(g) \in \langle a_k \rangle_{\sigma^{\alpha_k}}A$ is a good polynomial by Lemma \ref{PBWLemma3.4} (2). 
    
    If $x^{\beta}$ is a monomial such that $x^{\alpha_k} \preceq x^{\beta}$, for some $\beta \in \mathbb{N}^n $, we must find a good polynomial with leading monomial $x^{\beta}$. By assumption, if $a_k \in R$, there exists $g'=\phi(fg) \in \langle a_k\rangle_{\sigma^{\alpha_k}}A$ a good polynomial such that the leading monomial of $g'$ is $x_i$, for all $i$. Following the same argument of Proposition \ref{PBWProposition3.6}, we find a good polynomial $\overline{f}$ of $fA$ with leading monomial $x^{\beta}$ in at most $n\cdot {\rm max}\{|\beta_{i} - \alpha_{ki}| \}$ steps, proving that $A_A$ is a good module.
    \item[\rm (3)] Let $m=m_1x^{\alpha_1} + \cdots + m_kx^{\alpha_k} \in M\langle X \rangle_A$ be a good polynomial with leading coefficient $m_k\neq 0$. If ${\rm lm}(mx^{\alpha}r) \prec {\rm lm}(mx^{\alpha})$ for some $r \in R$, then $m_k\sigma^{\alpha_k}(\sigma^{\alpha}(r))=0$. Thus, ${\rm lm}(m\sigma^{\alpha}(r)) \prec {\rm lm}(m)$, which contradicts that $m$ is good polynomial. Hence, $mx^{\alpha}$ is also good polynomial for any $\alpha \in \mathbb{N}^n$. 
    
    \item[\rm (4)] Let $m=m_1x^{\alpha_1} + \cdots + m_kx^{\alpha_k} \in M \langle X \rangle_A$ be a good polynomial with leading coefficient $m_k\neq 0$. We claim that $mx^{\beta}$ is good, for any $\beta \in \mathbb{N}^n$. If $x^{\alpha_k + \beta}$ is the leading monomial of $mx^{\beta}$ and ${\rm lm}(mx^{\beta}r)\prec {\rm lm}(mx^{\beta})$, for some $r \in R$, then $m_k\sigma^{\alpha_k+ \beta}(r)=0$. By Lemma \ref{PBWLemma4.3} (1), ${\rm ann}_R(m_k)$ is $(\Sigma, \Sigma^{-1}, \Delta)$-stable, and so $r \in {\rm ann}_R(m_k)$. Furthermore, if ${\rm ann}_R(m_k)$ is $(\Sigma, \Sigma^{-1}, \Delta)$-stable, then $mx^{\beta}r = 0$ which proves that $mx^{\beta}$ is a good polynomial for any $\beta\in \mathbb{N}^n$. 
\end{itemize}
\end{proof}


Recall that $M_R$ is said to be a {\em uniform module} if every submodule of $N_R$ of $M_R$ is an essential submodule. Equivalently, $M_R$ is called {\em uniform} if the intersection of any two non-zero submodules of $M_R$ is non-zero \cite[p. 84]{Lam1998}. Theorem \ref{PBWTheorem4.6} characterizes the essential and uniform submodules of the induced module ${M}\langle X \rangle_A$.

\begin{theorem}\label{PBWTheorem4.6}
Let $A$ be bijective a skew PBW extension over $R$, $M_R$ be a right module, and $N_R$ be a submodule of $M_R$. If ${N}\langle X \rangle_A$ is a good module, then:
\begin{itemize}
    \item [\rm (1)] $N_R$ is essential in $M_R$ if and only if ${N}\langle X \rangle_A$ is essential in ${M}\langle X \rangle_A$.
    
    \item[\rm (2)] $N_R$ is uniform if and only if ${N}\langle X \rangle_A$ is uniform.
    \end{itemize}
\end{theorem}
\begin{proof}
\begin{itemize}
    \item[\rm (1)] If $T_R$ is a submodule of $M_R$ such that $T_R \cap N_R = 0$, we have that $T\langle X \rangle_A \cap N\langle X \rangle_A  = 0$, which proves that if $N\langle X \rangle_A$ is essential, then $N_R$ is essential. Suppose that $N_R$ is an essential submodule of $M_R$. By Lemma \ref{PBWLemma4.1} (1), we need to show that  $N\langle X \rangle_R$ is an essential submodule in $M\langle X \rangle_R$, that is, $mR \cap N\langle X \rangle_R \neq 0$, for any $m \in M\langle X \rangle_R$. We proceed by induction on the monomials. By Lemma \ref{PBWCorollary3.5}, we may assume that $m= m_1x^{\alpha_1} + \cdots + m_kx^{\alpha_k} \in M \langle X \rangle_R$ is a good polynomial. If $\alpha_k = 0$, then $m=m_k \in M_R$, and $ mR \cap N_R\neq 0$ because $N_R$ is essential. Suppose the statement is true for any leading monomial $x^{\beta}$ such that $x^{\beta}\prec x^{\alpha_k}$ with $\beta \in \mathbb{N}^n$. If $N_R$ is an essential submodule in $M_R$, then $m_kR \cap N_R \neq 0$, whence $m_k\sigma^{\alpha_k}(r) \in N_R$, for some $r \in R$. The element $m_k\sigma^{\alpha_k}(r) \in N \subseteq N\langle X \rangle_A$ is a good polynomial, and since $N\langle X \rangle_A$ is a good module, there exists a good polynomial $g \in m_k\sigma^{\alpha_k}(r)A \subseteq N\langle X \rangle_A$ such that ${\rm lm}(g) = {\rm lm}(mr)$. The leading coefficient of $g$ belongs to $m_k\sigma^{\alpha_k}(r)R$, so there exists $w \in R$ such that $mrw$ and $g$ have the same leading coefficient, and then $g$ and $mrw$ have the same leading term. If $g = mrw$, it follows that $g \in mR \cap N\langle X \rangle_R$ which proves that $N\langle X \rangle_R$ is essential in $M\langle X \rangle_R$. If $g \neq mrw$, then $mrw-g \neq 0$ and since $mrw$ and $g$ have the same leading term, the leading monomial of $mrw-g$ is $x^{\beta}$ for some $\beta \in \mathbb{N}^n$ with $x^{\beta} \prec x^{\alpha_k}$. Thus, by the inductive hypothesis, there exists $s \in R$ such that $h:= (mrw - g)s \in N\langle X \rangle_R$ with $h\neq 0$. Since $mrw$ and $g$ are good polynomials of the same leading term, they have the same annihilator in $R$. In this way, if $mrws = 0$, then $gs = 0$ and so $h=0$ which is a contradiction. Therefore, we have $mrws \neq 0$, and $0 \neq mrws = gs + h \in mR \cap N\langle X \rangle_R$. 
    
    \item[\rm (2)] It is clear that if $N\langle X \rangle_A$ is a uniform module, then $N_R$ is a uniform module. For the other implication, suppose that $N_R$ is a uniform module. If $N\langle X \rangle_A$ is not uniform, then there exist non-zero polynomials $f, g \in N\langle X \rangle_R$ such that $fA \cap gA = 0$ with $f=n_1x^{\alpha_1} + \cdots + n_kx^{\alpha_k}$ and $g=n_1'x^{\beta_1} + \cdots+ n_l'x^{\beta_l}$. By Lemma \ref{PBWCorollary3.5} (1), we may assume that $f$ and $g$ are good polynomials with ${\rm lm}(g)\preceq {\rm lm}(f)$. Since $N\langle X \rangle_A$ is a good submodule of $M\langle X \rangle_A$, there is a good polynomial $h \in gA$ with ${\rm lm} (h) = {\rm lm} (f)$. Let $n_k$ and $n_t''$ be the leading coefficients of the polynomials $f$ and $h$, respectively. Since $n_k, n_t'' \in N_R$ and $N_R$ is uniform, there exist $r, s \in R$ such that $n_kr = n_t''s \neq 0$. Consider the polynomial $z = f\sigma^{-\alpha_k}(r) - h\sigma^{-\alpha_k}(s) \in N\langle X \rangle_A$ with leading monomial $x^{\beta}$ for some $\beta \in \mathbb{N}$ where $x^{\beta} \prec x^{\alpha_k}$. If $f\sigma^{-\alpha_k}(r) =h\sigma^{-\alpha_k}(s) \in gA$, then $f\sigma^{-\alpha_k}(r)=0$ since $fA \cap gA = 0$. Thus, $f\sigma^{-\alpha_k}(r) \neq h\sigma^{-\alpha_k}(s)$ whence $z\neq 0$. Since $f$ and $g$ are good polynomials with ${\rm lm}(g)\preceq {\rm lm}(f)$, $zA \cap gA \neq 0$. So, there are $v_1, v_2 \in A$ such that $gv_2 = zv_1 = f\sigma^{-\alpha_k}(r)v_1 - h\sigma^{-\alpha_k}(s)v_1$. Since $f\sigma^{-\alpha_k}(r)$ and $h\sigma^{-\alpha_k}(s)$ are good polynomials of the same leading term and $f\sigma^{-\alpha_k}(r)v_1 = gv_2 + h\sigma^{-\alpha_k}(s)v_1 \in fA\cap gA = 0$, then they have the same annihilators in $A$, and hence $h\sigma^{-\alpha_k}(s)v_1 = 0$ which is a contradiction. Therefore, $N\langle X \rangle_A$ is a uniform submodule of $M\langle X \rangle_A$.
\end{itemize}
\end{proof}

\begin{corollary} [{\cite[Theorem 4.6]{LeroyMatczuk2004}}]
Let $S:= R[x;\sigma,\delta]$ and $N_R$ be a submodule of $M_R$ such that $\widehat{N}_S$ is good. Then:
\begin{itemize}
    \item [\rm (1)] $N_R$ is essential in $M_R$ if and only if $\widehat{N}_S$ is essential in $\widehat{M}_S$.
    
    \item[\rm (2)] $N_R$ is uniform if and only if $\widehat{N}_S$ is uniform.
\end{itemize}
\end{corollary}

We recall that a module $M_R$ has {\em finite uniform dimension} if there exist uniform submodules $U_1, \ldots , U_n$ of $M_R$ such that $U_1 \oplus \cdots \oplus U_n$ is an essential submodule of $M_R$ \cite[Definition 6.2]{Lam1998}. In this case, the {\em uniform dimension} of $M_R$ is denoted by ${\rm udim}(M_R) = n<\infty$. It is not difficult to see that a right module $M_R$ has infinite uniform dimension if and only if $M_R$ contains an infinite direct sum of non-zero submodules \cite[Proposition 6.4]{Lam1998}. 

Theorem \ref{PBWTheorem4.9} establishes sufficient conditions to guarantee that $M_R$ and ${M}\langle X \rangle_A$ have the same uniform dimension (c.f. \cite[Proposition 4.10]{Reyes2014}).

\begin{theorem}\label{PBWTheorem4.9}
Let $A$ be a bijective skew PBW extension over $R$ and $M_R$ be a right module. If ${M}\langle X \rangle_A$ is a good module, then ${\rm udim}({M}\langle X \rangle_A) = {\rm udim}(M_R)$.
\end{theorem}

\begin{proof}
Assume that ${\rm udim}(M_R) = k$. There exist $N_1, \ldots, N_k$ uniform submodules of $M_R$ such that $N_1 \oplus \cdots \oplus N_k$ is an essential submodule of $M_R$. By Theorem \ref{PBWTheorem4.6}, we have $N_1\langle X \rangle_A,\ldots, N_k\langle X \rangle_A$ are uniform submodules of $M\langle X \rangle_A$, and the submodule $N_1\langle X \rangle_A \oplus \cdots \oplus N_k\langle X \rangle_A$ is essential in $M\langle X \rangle_A$, whence ${\rm udim}(M\langle X \rangle_A)=k$.

If ${\rm udim}(M_R) = \infty$, there exist non-zero submodules $N_1, N_2, \ldots$ of $M_R$ such that $N_1 \oplus N_2 \oplus \cdots$ is a submodule of $M_R$. Thus, every $0 \neq N_i\langle X \rangle_A$ is a submodule of $M\langle X \rangle_A$ for every $i \ge 1$, and $N_1\langle X \rangle_A\oplus N_2\langle X \rangle_A\oplus \cdots$ is a submodule of $M\langle X \rangle_A$, which implies that ${\rm udim}(M\langle X \rangle_A)= \infty$. Therefore, ${\rm udim}(M\langle X \rangle_A)={\rm udim}(M_R)$.
\end{proof}

\begin{corollary}[{\cite[Theorem 4.9]{LeroyMatczuk2004}}] Let $S:= R[x;\sigma,\delta]$ and $M_R$ be a right module. If  $\widehat{M}_S$ is good, then ${\rm udim}(\widehat{M}_S) = {\rm udim}(M_R)$.
\end{corollary}

\section{Associated primes ideals of induced modules}\label{Associatedprimes}

In this section, we characterize the associated primes of induced modules over skew PBW extensions (c.f. \cite[Section 3]{NinoRamirezReyes}). We recall that a right module $N_R$ is called {\em prime} if $N_R \neq 0$ and ${\rm ann}_R(N_R) = {\rm ann}_R(N_R^{'})$, for every non-zero submodule $N_R^{'} \subseteq N_R$ \cite[Definition 1.1]{Annin2004}; an ideal $P$ of $R$ is said to be {\em associated} of $M_R$ if $P$ is prime and there exists a prime submodule $N_R \subseteq M_R$ such that $P = {\rm ann}_R(N_R)$. The set of associated prime ideals of $M_R$ is denoted by ${\rm Ass}(M_R)$ \cite[Definition 1.2]{Annin2004}. Following Leroy and Matczuk \cite{LeroyMatczuk2004}, it may happen that ${\rm Ass}(M_R)$ is not empty but ${\rm Ass}(N_R)=\emptyset$, for some non-zero submodule $N_R$ of $M_R$ \cite[p. 2756]{LeroyMatczuk2004}. For this reason, they worked with modules where ${\rm Ass}(N_R)$ is not empty for all non-zero submodule $N_R$ of $M_R$ and introduced the following definition: $M_R$ has {\em enough prime submodules} if any non-zero submodule $N_R$ of $M_R$ contains a prime submodule \cite[Definition 5.1]{LeroyMatczuk2004}. 

The following lemma shows that if $M_R$ has enough prime submodules, then any non-zero submodule of the induced module $M \langle X \rangle_A$ contains a good polynomial. Lemma \ref{PBWLemma5.4} generalizes  \cite[Lemma 5.4]{LeroyMatczuk2004}.

\begin{lemma}\label{PBWLemma5.4}
Let $A$ be a bijective skew PBW extension over $R$ and $M_R$ be a right module. If $M_R$ has enough prime submodules, then any non-zero submodule $N_A$ of $M \langle X \rangle_A$ contains a good polynomial $m$, with leading coefficient $m_k \neq 0$, such that $m_kR$ is a prime submodule of $M_R$.  
\end{lemma}
\begin{proof}
Let $N_A$ be a submodule of $M \langle X \rangle_A$ and  $n = n_1x^{\alpha_1}+ \cdots + n_kx^{\alpha_k} \in N_A$ be a non-zero polynomial with ${\rm lc}(n)=n_k\neq 0$ and minimal leading monomial in $nA$, that is, $x^{\alpha_k} \prec {\rm lm}(f)$, for every $f \in nA$. By Lemma $\ref{PBWLemma3.4}$, $n$ is a good polynomial, and since $M_R$ contains enough prime submodules, $n_kR$ contains a non-zero prime submodule $m_kR$ where $m_k=n_kr\neq 0$, for some $r \in R$. Thus, the polynomial $m = n\sigma^{-\alpha_k}(r) \in N_A$ is good with leading coefficient $m_k$ such that $m_kR$ is a prime submodule of $M_R$.
\end{proof}


Leroy and Matczuk characterized certain right annihilators of generated modules on $R[x;\sigma,\delta]$ where $\delta$ is a $\sigma$-derivation {\em $q$-quantized} of $R$ \cite[Lemma 5.6]{LeroyMatczuk2004}. Goodearl and Letzter \cite{GoodearlLetzter1994} introduced the notion of $q$-quantized derivation in the following way: a $\sigma$-derivation $\delta$ of $R$ is {\it $q$-quantized} if $\delta\sigma = q\sigma \delta$ where $q$ is a central, invertible element of $R$ such that $\sigma(q) = q$ and $\delta(q) = 0$. The ring $R[x;\sigma, \delta]$ is called a {\it $q$-skew polynomial ring} if $\delta$ is $q$-quantized \cite[p. 10]{GoodearlLetzter1994}. We consider the following definition for a finite family of endomorphims $\Sigma$ and a family of $\Sigma$-derivations $\Delta$ of $R$.

\begin{definition}
Let $R$ be a ring and $\Sigma$ be a finite family of endomorphisms of $R$. A family of $\Sigma$-derivations $\Delta$ of $R$ is called {\it quantized} if there exists $(q_1, \ldots, q_n)\in R^n$ such that $\delta_i \sigma_i= q_{i} \sigma_i \delta_i$, $\sigma_i(q_{j})=q_{j}$ and $\delta_{i}(q_{j})=0$, for every $0 \le i,j \le n$ where $q_i$ is a central and invertible element of $R$ for every $1\le i \le n$. 
\end{definition}

We say that a skew PBW extension $A$ over a ring $R$ is {\em quantized} if the family of $\Sigma$-derivations $\Delta$ defined in Proposition \ref{sigmadefinition} is quantized.

\begin{example}\label{ExampleQuantized} We present some examples of quantized skew PBW extensions.
\begin{enumerate}
    \item {\em The algebra of $q$-differential operators $D_{q,h}[x, y]$:} Let $q, h \in \Bbbk$, $q \neq 0$. We consider $\Bbbk[y][x; \sigma, \delta]$ with $\sigma(y) := qy$ and $\delta(y) := h$. By definition of skew polynomial ring, $xy = \sigma(y)x + \delta(y) = qyx + h$, and so $xy - qyx = h$. We can prove that $D_{q,h}[x, y] \cong \sigma(\Bbbk[y])\left \langle x \right \rangle$. It is not difficult to verify that $\delta \sigma = \sigma \delta$, whence $D_{q,h}[x, y]$ is a quantized skew PBW extension over $\Bbbk[y]$.
    \item {\em Additive analogue of the Weyl algebra}: Let $\Bbbk$ be a field and $A_n(q_1,\dots, q_n)$ be the $\Bbbk$-algebra generated by $x_1, \dots, x_n, t_1, \dots, t_n$ and subject to the relations:
\begin{align*}
    x_jx_i &= x_ix_j, \ \ \  t_jt_i = t_it_j, \ \ \ 1 \leq i, j \leq n. \\
x_jt_i &= t_ix_j, \ \ \ i\neq j. \\
x_it_i &= q_it_ix_i + 1, \ \ \ 1 \leq i \leq n.
\end{align*}
where $q_i \in \Bbbk \setminus \left \{0\right \}$. Thus, $A_n(q_1,\dots, q_n) \cong \sigma (\Bbbk[t_1,\dots ,t_n])\left \langle x_1,\dots, x_n \right \rangle$. Notice that $\sigma_i(t_i)=qt_i$ and $\delta_i(t_i)=1$, for all $1 \le i \le n$. Some simple computations show that $\delta_i\sigma_i =  \sigma_i\delta_i$, and so $A_n(q_1,\dots, q_n)$ is a quantized skew PBW extension over $\Bbbk[t_1,\dots ,t_n]$. 
\end{enumerate}
\end{example}

The following lemma characterizes the annihilators of generated modules by good polynomials and generalizes \cite[Lemma 5.6]{LeroyMatczuk2004}.

\begin{lemma}\label{PBWLemma5.6}
Let $A$ be a bijective skew PBW extension over $R$, $M_R$ be a right module, and $m= m_1x^{\alpha_1} + \cdots + m_kx^{\alpha_k}$ be a good polynomial of $M \langle X \rangle_A$ with leading coefficient $m_k\neq 0$. If $P:= {\rm ann}_R(m_kR)$, then:
\begin{itemize}
    \item [\rm (1)] If $P$ is $(\Sigma,\Sigma^{-1}, \Delta)$-stable, then ${\rm ann}_A(mA) = P\langle X \rangle$.
    \item [\rm (2)] If $A$ is quantized and $P$ is $(\Sigma,\Sigma^{-1})$-stable, then ${\rm ann}_A(mA) = P_{\Delta}\langle X \rangle$
    \item[\rm (3)] If $A$ is quantized, $m_kR$ is prime, and $mA$ contains good polynomial of any monomial greater than $x^{\alpha_k}$, then ${\rm ann}_A(mA)= \sigma^{-\alpha_k}(P_{\Sigma,\Delta})\langle X \rangle$.
\end{itemize}
\end{lemma}
\begin{proof}
\begin{itemize}
    \item[\rm (1)] Since $P$ is $\Delta$-stable, we have ${\rm ann}_A(mA) = P_{\Sigma}\langle X \rangle$ by Lemma \ref{PBWLemma3.7} (4). Since $(\Sigma,\Sigma^{-1})$-stable, we have $P_{\Sigma}=P$, and so ${\rm ann}_A(mA) = P\langle X \rangle$. 
    
    \item[\rm (2)] If $r \in P_{\Delta}$, then $\delta^{\beta}(r) \in P$, and so  $\sigma^{\alpha}(\delta^{\beta}(r)) \in P$ for every $\alpha,\beta \in \mathbb{N}^n$ by the $(\Sigma,\Sigma^{-1})$-stability of $P$. If $A$ is quantized, $\delta^{\beta}(\sigma^{\alpha}(r))=r_{\alpha,\beta}\sigma^{\alpha}(\delta^\beta(r)) \in P$ for some $r_{\alpha,\beta} \in R$ whence $P_{\Delta}$ is a $\Sigma$-invariant ideal of $R$. Hence, $P_{\Delta}\langle X \rangle$ is a two-sided ideal of $A$ and so $P_{\Delta}\langle X \rangle \subseteq {\rm ann}_A(mA)$. Let $f= b_1x^{\beta_1} + \cdots + b_tx^{\beta_t}$ be an element of ${\rm ann}_A(mA)$. By induction on the monomials, we show that $\delta^{\theta}(b_i) \in P$, for any $\theta \in \mathbb{N}^n$ and $1 \le i \le t$. Since $m$ is a good polynomial and $P$ is a $(\Sigma, \Sigma^{-1})$-stable ideal of $R$, ${\rm ann}_A(mR)=\sigma^{-\alpha_k}(P)\langle X \rangle=P\langle X \rangle$ by Lemma \ref{PBWLemma3.4}, and thus $b_i \in P$ for every $1 \le i \le t$. Assume that for any leading monomial $x^{\gamma}$ with $x^ {\gamma}\prec x^{\theta}$, we have $\delta^{\gamma}(b_i)\in P$. If $f \in {\rm ann}_A(mA)$ and ${\rm ann}_A(mR)=P\langle X \rangle_A$, then $mrx^{\theta}f=0$. Since $A$ is quantized, it follows that $x^{\theta}b_i= r_{1}x^{\theta_{1}} + \cdots + r_sx^{\theta_{s}}$ where each $r_j$ is a finite sum of several evaluations of $\sigma^{\theta_{j}}$'s and $\delta^{\theta - \theta_{j}}$'s in the element $b_i$, for every $1\le j \le s$. Thus, if $P$ is a $(\Sigma, \Sigma^{-1})$-stable ideal and $\delta^{\gamma}(b_i) \in P$ for any $\gamma\in \mathbb{N}^n$ with $x^ {\gamma}\prec x^{\theta}$, then $mRr_ix^{\theta_j}=0$ for all $1 \le j \le s$ where $\theta_j \neq 0$. In this way, we have $mrx^{\theta}b_i = mr\delta^{\theta}(b_i)$, and hence $mrx^{\theta}f= mr(\delta^{\theta}(b_1)x^{\beta_1} + \cdots + \delta^{\theta}(b_t)x^{\beta_t})$ which shows that $\delta^{\theta}(b_i) \in P$ for all $1 \le i \le t$, and therefore $f \in P_{\Delta}\langle X\rangle$.
    
    \item[\rm (3)] Let $f=b_1x^{\beta_1} + \cdots + b_tx^{\beta_t} \in {\rm ann}_A(mA)$. We prove by induction on the leading monomials that $\delta^{\theta}(b_i)\in \sigma^{-\alpha_k}(P_{\Sigma})$, for all $\theta \in \mathbb{N}^n$ and $1 \le i \le t$. Since $mA$ contains good polynomials of any monomial greater than $x^{\alpha_k}$, there exists $f_{\gamma_i} \in mA$ with leading monomial $x^{\gamma_i}$, for each $x^{\gamma_i}\succeq x^{\alpha_k}$. Let $m_{\gamma_i}$ be the leading coefficient of $f_{\gamma_i}$ where $m_{\gamma_i} \in m_kR$. If $m_kR$ is prime, then $m_{\gamma_i}R$ is prime, and thus ${\rm ann}(m_{\gamma_i}R)= P$ for all $\gamma_i$. Since the $f_{\gamma_i}$ are good polynomials, ${\rm ann}_A (f_{\gamma_i}R)=\sigma^{-\gamma_i}(P)\langle X \rangle$ by Lemma \ref{PBWLemma3.4}. If $E$ denotes the submodule of $M \langle X \rangle _R$ defined by $E= \sum_{\gamma_i \ge \alpha_k} f_{\gamma_i}R$, then $E \subseteq mA$ and ${\rm ann}_A (mA)\subseteq {\rm ann}_A (E)=\bigcap_{\gamma_i \ge \alpha_k} \sigma^{-\gamma_i}(P)\langle X \rangle = \sigma^{-\alpha_k}(P_{\Sigma})\langle X \rangle$. In this way, if $f \in {\rm ann}_A(mA)$, then $f \in \sigma^{-\alpha_k}(P_{\Sigma})\langle X \rangle$ and thus $b_i \in \sigma^{-\alpha_k}(P_{\Sigma})$. 
    
    Now, assume that for any leading monomial $x^{\gamma}$ with $x^ {\gamma}\prec x^{\theta}$, we have $\delta^{\gamma}(b_i)\in \sigma^{-\alpha_k}(P_{\Sigma})$. If $Ex^{\theta}\subseteq mA$ and $f=b_1x^{\beta_1} + \cdots + b_tx^{\beta_t} \in {\rm ann}_A(mA)$, then $Ex^{\theta}f = 0$. In addition if $A$ is quantized, then $x^{\theta}b_i= r_{1}x^{\theta_{1}} + \cdots + r_sx^{\theta_{s}}$ where each $r_j$ is a finite sum of several evaluations of $\sigma^{\theta_{j}}$'s and $\delta^{\theta - \theta_{j}}$'s in the element $b_i$, for every $1\le j \le s$. If $\delta^{\gamma}(b_i)\in \sigma^{-\alpha_k}(P_{\Sigma})$ for any $\gamma\in \mathbb{N}^n$ with $x^ {\gamma}\prec x^{\theta}$, then $Er_ix^{\theta_j}=0$ for all $1 \le j \le s$ where $\theta_j \neq 0$. Thus $Ex^{\theta}b_i = E\delta^{\theta}(b_i)$, and hence $Ex^{\theta}f= E(\delta^{\theta}(b_1)x^{\beta_1} + \cdots + \delta^{\theta}(b_t)x^{\beta_t})$ whence $\delta^{\theta}(b_i) \in \sigma^{-\alpha_k}(P_{\Sigma})$ for all $1 \le i \le t$. So ${\rm ann}_A(mA) \subseteq (\sigma^{-\alpha_k}(P_{\Sigma}))_{\Delta}A$, and since $A$ is quantized, it follows that $(\sigma^{-\alpha_k}(P_{\Sigma}))_{\Delta} = \sigma^{-\alpha_k}(P_{\Sigma, \Delta})$, and hence ${\rm ann}_A(mA) \subseteq \sigma^{-\alpha_k}(P_{\Sigma, \Delta})$.
   
    To prove the other inclusion, if $r \in \sigma^{-\alpha_k}(P_{\Sigma, \Delta})$, then $\sigma^{\alpha_k}(r) \in P_{\Sigma, \Delta}$, and thus $\sigma^{\alpha_k}(\sigma^{\alpha}\delta^{\beta}(r))$ for all $\alpha,\beta \in \mathbb{N}$. Since $m$ is a good polynomial of leading monomial $x^{\alpha_k}$ and leading coefficient $m_k \neq 0$, then $mR\sigma^{\alpha}\delta^{\beta}(r)=0$. This implies that $mRx^{\gamma}r=0$, for any $\gamma \in \mathbb{N}^n$, and so $\sigma^{-\alpha_k}(P_{\Sigma, \Delta}) \subseteq {\rm ann}_A(mA)$. 
\end{itemize}
\end{proof}


The following theorem characterizes the associated prime ideals of $M\langle X \rangle_A$ where $M_R$ is a right module that contains enough prime submodules. 

\begin{theorem}\label{PBWTheorem5.7}
Let $A$ be a bijective skew PBW extension over $R$, $M_R$ be a right module, and $Q \in {\rm Ass}(M\langle X \rangle_A)$. If $M_R$ contains enough prime submodule, then:
\begin{itemize}
    \item [\rm (1)] If $P$ is $(\Sigma,\Sigma^{-1},\Delta)$-stable for every $P \in {\rm Ass}(M_R)$, then $Q = P\langle X \rangle$ for some $P \in {\rm Ass}(M_R)$.
    
    \item[\rm (2)] If $A$ is quantized and $P$ is $(\Sigma,\Sigma^{-1})$-stable for every $P \in {\rm Ass}(M_R)$, then $Q=P_{\Delta}\langle X \rangle$ for some $P \in {\rm Ass}(M_R)$.
    
    \item[\rm (3)] If $A$ is quantized and the module $M\left \langle X \right \rangle_A$ is good, then $Q = P_{\Sigma,\Delta}\langle X \rangle$ for some $P \in {\rm Ass}(M_R)$ and $P_{\Sigma,\Delta}$ is $(\Sigma,\Sigma^{-1})$-stable.
\end{itemize}
\end{theorem}
\begin{proof}
\begin{itemize}
    \item[\rm (1)] Let $N_A$ be a prime submodule of $M\langle X \rangle_A$ such that ${\rm ann}_A(N) = Q$. By Lemma \ref{PBWLemma5.4}, there is a good polynomial $m=m_1x^{\alpha_1}+ \cdots+ m_kx^{\alpha_k} \in N_A$ with leading coefficient $m_k \neq 0$ such that $m_kR$ is prime submodule of $M_R$. Since $N_A$ is prime, then $mA$ is also prime and ${\rm ann}_A(mA) = {\rm ann}_A(N) = Q$. By Lemma \ref{PBWLemma5.6} (1), we have $Q = P\langle X \rangle$ with $P={\rm ann}_R(m_kR) \in {\rm Ass}(M_R)$.
    
    \item[(2)] In the same way, if $N_A$ is a prime submodule of $M\langle X \rangle_A$ with ${\rm ann}_A(N) = Q$, there exists a good polynomial $m=m_1x^{\alpha_1}+ \cdots+ m_kx^{\alpha_k} \in N_A$ with leading coefficient $m_k \neq 0$ such that $m_kR$ is prime submodule of $M_R$ by Lemma \ref{PBWLemma5.4}. Since $N_A$ is prime, then $mA$ is also prime and ${\rm ann}_A(mA) = {\rm ann}_A(N) = Q$. By Lemma \ref{PBWLemma5.6} (2), $Q = P_{\Delta}\langle X \rangle$ where $P={\rm ann}_R(m_kR) \in {\rm Ass}(M_R)$.
    
    \item[(3)] If $N_A$ is a prime submodule of $M\langle X \rangle_A$ such that ${\rm ann}_A(N) = Q$, there exists a good polynomial $m=m_1x^{\alpha_1}+ \cdots+ m_kx^{\alpha_k} \in N_A$ with leading coefficient $m_k \neq 0$ by Lemma \ref{PBWLemma5.4}. If $M\langle X \rangle_A$ is good, $mA$ contains a good polynomial $m'$ with leading monomial $x^{\beta}$ such that $x^{\alpha_k} \preceq x^{\beta}$. Additionally, $mA$ is prime module which implies that ${\rm ann}_A(mA) = Q$, and by Lemma \ref{PBWLemma5.6} (3), ${\rm ann}_A(m'A) = \sigma^{-\beta}(P_{\Sigma,\Delta})\langle X \rangle$ and ${\rm ann}_A(mA) = \sigma^{-\alpha_k}(P_{\Sigma,\Delta})\langle X \rangle$. Thus, $P_{\Sigma,\Delta}$ is $(\Sigma,\Sigma^{-1})$-stable, and hence $Q = {\rm ann}_A(mA) = P_{\Sigma,\Delta}\langle X \rangle$.
\end{itemize}
\end{proof}

\begin{corollary}[{\cite[Theorem 5.7]{LeroyMatczuk2004}}]\label{Theorem5.7Leroy} Let $S:= R[x;\sigma,\delta]$ and $M_R$ be a right module. If $M_R$ contains enough prime submodule and  $Q \in {\rm Ass}(\widehat{M}_S)$, then:
\begin{itemize}
    \item [\rm (1)] If for every $P \in {\rm Ass}(M_R)$, $\sigma(P) = P$ and $\delta(P) \subseteq P$, then $Q = PS$ for some $P \in {\rm Ass}(M_R)$.
    \item[\rm (2)] If $\delta$ is $q$-quantized and $\sigma(P) = P$ for all $P \in {\rm Ass}(M_R)$, then $Q=P_{\delta}S$ for some $P \in {\rm Ass}(M_R)$.
    \item[\rm (3)] If $\delta$ is $q$-quantized and $\widehat{M}_S$ is a good module, then $Q = P_{\sigma, \delta}S$ for some $P \in {\rm Ass}(M_R)$ and $\sigma(P_{\sigma, \delta}) = P_{\sigma,\delta}$.
\end{itemize}
\end{corollary}

Leroy and Matczuk presented an example where Corollary \ref{Theorem5.7Leroy} fails if the module $M_R$ does not have enough prime submodules. This shows that this hypothesis is not superfluous. If $M$ is a $\Bbbk$-linear space with basis $\{v_i\}_{i\in \mathbb{Z}}$ and $R = \Bbbk\langle X \rangle$ is the free algebra over $\Bbbk$ on the set $X = \{x_i\}_{i\in \mathbb{Z}}$, then $M$ has a module structure over $R$ given by $v_ix_k=v_{i+1}$ if $i \le k$ and $0$ otherwise. Let $\sigma$ be the automorphism of $R$ defined by $\sigma(x_k) = x_{k+1}$ for any $k \in \mathbb{Z}$ and $S := R[t, \sigma]$. They showed that ${\rm Ass}(M_R) = \emptyset$ and that $\widehat{M}_S$ is prime with ${\rm Ass}(\widehat{M}_S) = \{0\}$ \cite[Example 5.8]{LeroyMatczuk2004}.

The following theorem shows when $M\langle X \rangle_A$ is a prime module and characterizes its associated prime ideals in terms of the associated primes of $M_R$. 

\begin{theorem}\label{PBWTheorem5.10}
Let $A$ be a bijective skew PBW extension over $R$ and $M_R$ be a right module. If $M_R$ is a prime module with $P = {\rm ann}_R(M)$, then:
\begin{itemize}
    \item [\rm (1)] If $P$ is $(\Sigma,\Sigma^{-1},\Delta)$-stable, then the induced module $M\langle X \rangle_A$ is prime with the associated prime ideal equal to $Q=P\langle X \rangle$.
    
    \item[\rm (2)] If $A$ is quantized and $P$ is $(\Sigma,\Sigma^{-1})$-stable, then $M\langle X \rangle_A$ is prime with the associated prime ideal equal to $Q=P_{\Delta}\langle X \rangle$.
    
    \item[\rm (3)] If $A$ is quantized and the module $M\langle X \rangle_A$ is good, then $M\langle X \rangle_A$ is a prime module if and only if $P_{\Sigma,\Delta}$ is $(\Sigma,\Sigma^{-1})$-stable. If $M\langle X \rangle_A$ is prime, then its associated prime ideal is equal to $Q=P_{\Sigma,\Delta}\langle X \rangle$.
\end{itemize}
\end{theorem}
\begin{proof}
\begin{itemize}
    \item[\rm (1)] Let $N_A$ be a submodule of $M\langle X \rangle_A$ and $m=m_1x^{\alpha_1} + \cdots + m_kx^{\alpha_k}$ be an element of $N_A$ with minimal leading monomial in $mA_A$, i.e., $x^{\alpha_k}\preceq {\rm lm}(f)$ for all $f \in mA_A$. By Lemma \ref{PBWLemma3.4}, $m$ is a good polynomial and since $M_R$ is prime, then $m_kR$ is a submodule prime of $M_R$ with $P= {\rm ann}_R(m_kR)$ by Lemma \ref{PBWLemma5.4}. Additionally, by Lemma \ref{PBWLemma5.6} (1), we have $P\langle X \rangle={\rm ann}_A (mA)$, and thus $P\langle X \rangle \subseteq {\rm ann}_A(M\langle X \rangle) \subseteq {\rm ann}_A(N) \subseteq {\rm ann}_A (mA) = P\langle X \rangle$. This implies that ${\rm ann}_A(N) = P\langle X \rangle$ for any submodule $N_A$ of $M\langle X \rangle_A$, whence $M\langle X \rangle_A$ is a prime module with associated prime ideal $Q=P\langle X \rangle$.
    
    \item[(2)] Following the same argument in (1) and Lemma  \ref{PBWLemma5.6} (2), $M\langle X \rangle_A$ is a prime module with associated prime ideal $Q=P_{\Delta}\langle X \rangle$.
    
    \item[(3)] In the same way, the argument of (1) and Lemma \ref{PBWLemma5.6} (3) show that $M\langle X \rangle_A$ is a prime module with associated prime ideal $Q=P_{\Sigma,\Delta}\langle X \rangle$ and $P_{\Sigma,\Delta}$ is $(\Sigma,\Sigma^{-1})$-stable.
\end{itemize}
\end{proof}

\begin{corollary}[{\cite[Theorem 5.10]{LeroyMatczuk2004}}] Let $S:= R[x;\sigma,\delta]$ and $M_R$ be a right module. If $M_R$ is prime with $P = {\rm ann}_R(M)$, then:
\begin{itemize}
    \item [\rm (1)] Suppose that $\sigma(P)=P$ and $\delta(P) \subseteq P$. Then the induced module $\widehat{M}_S$ is prime with the associated prime ideal equal to $PS =Q$.
    \item[\rm (2)] Suppose that $\delta$ is a $q$-quantized $\sigma$-derivation and $\sigma(P) = P$. Then $\widehat{M}_S$ is prime with the associated prime ideal equal to $P_{\delta}S = Q$.
    \item[\rm (3)] Suppose that $\delta$ is a $q$-quantized $\sigma$-derivation and the module $\widehat{M}_S$ is good. Then $\widehat{M}_S$ is a prime module if and only if $\sigma(P_{\sigma, \delta}) = P_{\sigma, \delta}$. Moreover, if $\widehat{M}_S$ is prime, then its associated prime ideal is equal to $P_{\sigma,\delta}S = Q$.
\end{itemize}
\end{corollary}

\section{Examples}\label{examplespaper}
 
The importance of our results is appreciated when we apply them to algebraic structures more general than those considered by Leroy and Matczuk \cite{LeroyMatczuk2004}, that is, some noncommutative rings which cannot be expressed as skew polynomial rings. In this section, we consider several families of rings that have been studied in the literature which are subfamilies of skew PBW extensions.

\begin{example}
   \cite[p. 30]{LFGRSV} The {\em diffusion algebra} $A$ is generated by $2n$ indeterminates  $D_i, x_i$ over $\Bbbk$ with $1 \leq i \leq n$ and subjects to the relations
\begin{center}
    $\displaystyle x_ix_j = x_jx_i,\ \ x_iD_j = D_jx_i, \ \ 1 \leq i, j \leq n$, \\
$c_{ij}D_iD_j - c_{ji}D_jD_i = x_jD_i - x_iD_j , \ \  i < j, \ \ c_{ij} , c_{ji} \in \Bbbk^{*}$.
\end{center}
According to Definition \ref{def.skewpbwextensions}, the algebra $A$ can be seen as a skew PBW extension over  $\Bbbk[x_1, \dots , x_n]$, but not as a PBW extension or an iterated skew polynomial ring of injective type. If $M_{R}$ is a right module over $R:=\Bbbk[x_1, \dots , x_n]$ and $M\langle D_1, \ldots,D_n \rangle_A$ is a good module, then Theorem \ref{PBWTheorem4.9} shows that $M_R$ and $M\langle D_1, \ldots,D_n \rangle_A$ have the same uniform dimension. If $M_{R}$ contains enough prime submodules, then Theorem \ref{PBWTheorem5.7} characterizes the associated prime ideals of $M\langle D_1, \ldots,D_n \rangle_A$, and if $M_{R}$ is a prime module with $P={\rm ann}_{R}(M)$, then $M\langle D_1, \ldots,D_n \rangle_A$ is prime module with associated prime ideal $P\langle D_1, \ldots,D_n \rangle$ by Theorem \ref{PBWTheorem5.10}.
\end{example}

\begin{example} \cite[Section 25.2]{BurdikNavratil2009} 
   The basis of the Lie algebra $\mathfrak{g} = \mathfrak{so}(5, \mathbb{C})$ consists of the elements $\bf{J}_{\alpha \beta} = -{\bf J}_{\beta \alpha}$, $\alpha, \beta = 1, 2, 3, 4, 5$ satisfying the commutation relations $[\bf{J}_{\alpha \beta}, \bf{J}_{\mu \nu}] = \delta_{\beta \mu} \bf{J}_{\alpha \nu} + \delta_{\alpha \nu}\bf{J}_{\beta \mu} - \delta_{\beta \nu}\bf{J}_{\alpha \mu} - \delta_{\alpha \mu}\bf{J}_{\beta \nu}$. Having in mind the classical PBW theorem for the universal enveloping algebra $U(\mathfrak{so}(5, \mathbb{C}))$ of $\mathfrak{so}(5, \mathbb{C})$, and since $U(\mathfrak{so}(5, \mathbb{C}))$ is a PBW extension of $\mathbb{C}$ \cite[Section 5]{BellGoodearl1988}, then $U(\mathfrak{so}(5, \mathbb{C}))$ is a skew PBW extension over $\mathbb{C}$, i.e., $U(\mathfrak{so}(5, \mathbb{C})) \cong \sigma(\mathbb{C})\langle {\bf J}_{\alpha \beta}\mid 1 \le \alpha \leq \beta \le 5 \rangle$.  If $M_{\mathbb{C}}$ is a right module over $\mathbb{C}$ and $M\langle {\bf J}_{\alpha \beta} \rangle_A$ is a good module, then Theorem \ref{PBWTheorem4.9} characterizes the uniform dimension of the module $M\langle {\bf J}_{\alpha \beta} \rangle_A$ over $A:= U(\mathfrak{so}(5, \mathbb{C}))$, and if $M_{\mathbb{C}}$ contains enough prime submodules, then Theorem \ref{PBWTheorem5.7} describes the associated prime ideals of $M\langle {\bf J}_{\alpha \beta} \rangle_A$. 
\end{example}

\begin{example}
         Following Havli\v{c}ek et al. \cite[p. 79]{HavlicekKlimykPosta2000}, the $ \mathbb{C}$-algebra $U_q'(\mathfrak{so}_3)$ is generated by the indeterminates $I_1, I_2$, and $I_3$, subject to the relations given by
    \begin{equation*}
        I_2I_1 - qI_1I_2 = -q^{\frac{1}{2}}I_3,\quad 
        I_3I_1 - q^{-1}I_1I_3 = q^{-\frac{1}{2}}I_2,\quad 
        I_3I_2 - qI_2I_3 = -q^{\frac{1}{2}}I_1,
    \end{equation*}
    
    where $0 \neq q \in \mathbb{C}$. It is straightforward to show that $U_q'(\mathfrak{so}_3)$ cannot be expressed as an iterated skew polynomial ring. However, this algebra can be seen as a skew PBW extension over $\mathbb{C}$ \cite[Example 1.3.3]{LFGRSV}. If $M_{\mathbb{C}}$ is a right module over $\mathbb{C}$ and $M\langle I_1, I_2, I_3 \rangle_A$ is a good module, then Theorem \ref{PBWTheorem4.9} characterizes the uniform dimension of $M\langle I_1, I_2, I_3 \rangle_A$ with $A:=U_q'(\mathfrak{so}_3)$. If $M_{\mathbb{C}}$ contains enough prime submodules, then Theorem \ref{PBWTheorem5.7} describes the associated prime ideals of $M\langle I_1, I_2, I_3 \rangle_A$, and if $M_{\mathbb{C}}$ is a prime module with $P={\rm ann}_{\mathbb{C}}(M)$, then $M\langle I_1, I_2, I_3 \rangle_A$ is a prime module with associated prime ideal $P\langle I_1, I_2, I_3 \rangle$ by Theorem \ref{PBWTheorem5.10}.
\end{example}

\begin{example}
    Zhedanov \cite{Zhedanov} defined the {\em Askey-Wilson algebra} AW(3) as the algebra generated by three indeterminates $K_0, K_1$, and $K_2$, subject to the commutation relations given by
     \begin{align*}
         e^{\omega}K_0K_1 - e^{-\omega}K_1K_0 &= K_2, e^{\omega} K_2K_0 - e^{-\omega}K_0 K_2 = BK_0 + C_1K_1 + D_1, \\
          e^{\omega}K_1K_2 - e^{-\omega}K_2K_1 &= BK_1 + C_0K_0 + D_0, 
     \end{align*}
     
     where $B, C_0, C_1,D_0, D_1 \in \mathbb{R}$ and $\omega$ is an arbitrary real parameter. It is not difficult to see that AW(3) cannot be expressed as an iterated skew polynomial ring. On the other hand, using techniques such as those presented in \cite[Theorem 1.3.1]{LFGRSV}, it can be shown that ${\rm AW(3)}$ is a skew PBW extension of endomorphism type over $\mathbb{R}$, that is,  ${\rm AW(3)} \cong \sigma(\mathbb{R})\langle K_0, K_1, K_2\rangle$. If $M_{\mathbb{R}}$ is a right module over $\mathbb{R}$ and $M\langle K_0, K_1, K_2 \rangle_A$ is a good module over $A:=AW(3)$, then ${\rm udim}(M\langle K_0,K_1,K_2  \rangle_A)={\rm udim}(M_{\mathbb{R}})$ by Theorem \ref{PBWTheorem4.9}. If $M_{\mathbb{R}}$ contains enough prime submodules, then Theorem \ref{PBWTheorem5.7} characterizes the associated prime ideals of $M\langle K_0,K_1,K_2 \rangle_A$, and if $M_{\mathbb{R}}$ is a prime module with $P={\rm ann}_{\mathbb{R}}(M)$, then $M\langle K_0,K_1,K_2 \rangle_A$ is a prime module with associated prime ideal $P\langle K_0,K_1,K_2 \rangle$ by Theorem \ref{PBWTheorem5.10}.
\end{example}

\section{Future work}\label{futurework}

As a possible future work, we have in mind to study the {\em couniform dimension} of modules introduced by Varadarajan \cite{Sarathetal1979, Varadarajan1979}. Annin \cite{Annin2005} studied this dimension on the inverse polynomial module $M[x^{-1}]$ over $R[x;\sigma]$. We think that a natural task is to investigate the counifom dimension of the polynomial module $M[x^{-1}]$ on structures more general than skew polynomial rings of automorphism type.

Macdonald \cite{Macdonald1973} introduced the dual notion of primary decomposition known as {\em secondary representation}, where the main ideals of his theory are called {\it attached primes}. Since Annin \cite{Annin2011} characterized the attached prime ideals on the inverse polynomial module $M[x^{-1}]$ over $R[x;\sigma]$, another natural task is to investigate the attached prime ideals of modules $M[x^{-1}]$ over structures more general than skew polynomial rings of automorphism type (for instance, the skew PBW extensions and the {\em semi-graded rings} introduced by Lezama and Latorre \cite{LezamaLatorre2017}).



\begin{thebibliography}{35}

\bibitem{AbdiTalebi2024} M. Abdi and Y. Talebi, On the diameter of the zero-divisor graph over skew PBW extensions, J. Algebra Appl. {\bf 23}(05) 2450089.

\bibitem{LezamaAcostaReyes2015}J. P. Acosta, O. Lezama and A. Reyes, Prime ideals of skew PBW  extensions, {\em Rev. Un. Mat. Argentina} {\bf 56}(2) (2015) 39--55.

\bibitem{annin}S. Annin, {\em Associated and Attached Primes Over Noncommutative Rings}, PhD Thesis, University of California, Berkeley (2002).

\bibitem{Annin2002}S. Annin, Associated primes over skew polynomial rings, {\em Comm. Algebra} {\bf 30}(5) (2002) 2511--2528.

\bibitem{Annin2004}S. Annin, Associated primes over Ore extension rings, {\em J. Algebra Appl.} {\bf 3}(2) (2004) 193--205.

\bibitem{Annin2005}{S. Annin}, Couniform dimension over skew polynomial rings, {\em Comm. Algebra} {\bf 33}(4) (2005) 1195--1204. 

\bibitem{Annin2011} S. Annin, Attached primes under skew polynomial extensions, {\em J. Algebra Appl.} {\bf10}(3) (2011) 537--547.






\bibitem{Bavula2021}V. V. Bavula, Description of bi-quadratic algebras on 3 generators with PBW basis, {\em J. Algebra} {\bf 631} (2023) 695--730.





\bibitem{BellGoodearl1988}A. Bell and K. Goodearl, Uniform rank over differential operator rings and Poincar\'e-Birkhoff-Witt extensions, {\em Pacific J. Math.} {\bf 131}(1) (1988) 13--37.

\bibitem{BellSmith1990}A. D. Bell and S. P. Smith, Some 3-dimensional skew polynomial rings. University of Wisconsin, Milwaukee, preprint, (1990).

\bibitem{BrewerHeinzer1974}J. Brewer and W. Heinzer, Associated primes of principal ideals, {\em Duke Math. J.} {\bf 41}(1) (1974) 1--7.



\bibitem{BurdikNavratil2009} \v{C}. Burd\'ik and O. Navr\'atil, Decomposition of the Enveloping Algebra so(5). In:  Silvestrov, S., Paal, E., Abramov, V., Stolin, A., eds. {\em Generalized Lie Theory in Mathematics, Physics and Beyond}. Springer, Berlin, Heidelberg, (2009) 297--302.

\bibitem{Faith2000}C. Faith, Associated primes in commutative polynomial rings, {\em Comm. Algebra} {\bf 28}(8) (2000) 3983--3986.


\bibitem{LFGRSV}W. Fajardo, C. Gallego, O. Lezama, A. Reyes, H. Su\'arez, H. Venegas, {\em Skew PBW Extensions: Ring and Module-theoretic properties, Matrix and Gr\"obner Methods, and Applications}, Algebra and Applications. Springer, Cham, 2020.

\bibitem{GallegoLezama2011}C. Gallego and O. Lezama, Gr\"obner bases for ideals of $\sigma$-PBW extensions, {\em Comm. Algebra} {\bf 39}(1) (2011) 50--75.



\bibitem{GomezTorrecillas2014} J. Gomez Torrecillas, Basic Module Theory over Non-commutative Rings with Computational Aspects of Operator Algebras. In: Barkatou, M.,  Cluzeau, T., Regensburger, G., Rosenkranz, M., eds. {\em Algebraic and Algorithmic Aspects of Differential and Integral Operators. AADIOS 2012.} Lecture Notes in Computer Science, Vol. 8372, (Springer, Berlin, Heidelberg, 2014) pp. 23--82.



\bibitem{GoodearlLetzter1994} K. R. Goodearl and E. S. Letzter, {\em Prime Ideals in Skew and $q$-Skew Polynomial Rings} Vol. 521, American Mathematical Soc. (1994).








\bibitem{HavlicekKlimykPosta2000}M. Havl\'i\v{c}ek, A. U. Klimyk and S. Po\v{s}ta, Central elements of the algebras $U'(\mathfrak{so}_m)$ and $U(\mathfrak{iso}_m)$, {\em Czech. J. Phys.} {\bf 50}(1) (2000) 79--84.


\bibitem{HigueraReyes2022} S. Higuera and A. Reyes, On Weak annihilators and Nilpotent Associated Primes of Skew PBW Extensions, {\em Comm. Algebra} {\bf 51}(11) (2023) 4839--4861.







\bibitem{Jordan2000}D. A. Jordan, Down-Up Algebras and Ambiskew Polynomial Rings, {\em J. Algebra} {\bf 228}(1) (2000) 311--346.



\bibitem{KandryWeispfenninig1990}A. Kandri-Rody and V. Weispfenning, Non-commutative Gr\"obner Bases in Algebras of Solvable Type, {\em J. Symbolic Comput.} {\bf9}(1) (1990) 1--26.


\bibitem{Lam1998}T. Y. Lam,  Lectures on Modules and Rings, {\em Graduate Texts in Mathematics} Vol. 189, Springer-Verlag, Berlin, (1998).


\bibitem{LezamaLatorre2017}E. Latorre. and O. Lezama, Non-commutative algebraic geometry of semi-graded rings, {\em Internat. J. Algebra Comput.} {\bf27}(4) (2017) 361--389.

\bibitem{LeroyMatczuk2004}A. Leroy and J. Matczuk, On Induced Modules Over Ore Extensions, {\em Comm. Algebra} {\bf32}(7) (2004) 2743--2766.


\bibitem{LezamaReyes2014}O. Lezama and A. Reyes, Some Homological Properties of Skew PBW Extensions, {\em Comm. Algebra} {\bf42}(3) (2014) 1200--1230.




\bibitem{LouzariReyes2019}M. Louzari and A. Reyes, Minimal prime ideals of skew PBW extensions over 2-primal compatible rings, {\em Rev. Colombiana Mat.} {\bf 54}(1) (2020) 39--63.

\bibitem{LunqunJingwang2011} O. Lunqun and L. Jingwang, On weak $(\alpha, \delta)$-compatible rings, {\em Internat. J. Algebra Comput.} {\bf 5}(26) (2011) 1283--1296.

\bibitem{Macdonald1973} I. G. Macdonald, Secondary representation of modules over a commutative ring, {\em Sympos. Math.} {\bf11} (1973) 23--43.


\bibitem{McConnellRobson2001}J. McConnell and J. Robson, {\em Noncommutative Noetherian Rings}, Graduate Studies in Mathematics AMS (2001).


\bibitem{NinoReyes2019}A. Ni\~no and A. Reyes, Some remarks about minimal prime ideals of skew Poincar\'e-Birkhoff-Witt extensions, {\em Algebra Discrete Math.} {\bf 30}(2) (2019) 207--229.

\bibitem{NinoRamirezReyes}A. Ni\~{n}o, M. C. Ram\'irez and A. Reyes, Associated prime ideals over skew PBW extensions, {\em Comm. Algebra} {\bf 48}(12) (2020) 5038--5055.


\bibitem{Ore1933}O. Ore, Theory of Non-Commutative Polynomials, {\em Ann. of Math. {\rm(}2{\rm)}} {\bf34}(3) (1933) 480-508.

\bibitem{OuyangBirkenmeier2012}L. Ouyang and G. F. Birkenmeier, Weak annihilator over extension rings, {\em Bull. Malays. Math. Sci. Soc.} {\bf35}(2) (2012) 345--347.



\bibitem{Reyes2014}A. Reyes, Uniform dimension over skew PBW extensions, {\em Rev. Colombiana Mat.} {\bf48}(1) (2014) 79--96.

\bibitem{Reyes2015}A. Reyes, Skew PBW extensions of Baer, quasi-Baer, p.p. and p.q.-rings, {\em Rev. Integr. Temas Mat.} {\bf33}(2) (2015) 173--189.

\bibitem{Reyes2019}A. Reyes, Armendariz modules over skew PBW extensions, {\em Comm. Algebra} {\bf47}(3) (2019) 1248--1270.






\bibitem{ReyesSuarez2019CMS}A. Reyes and H. Su\'arez, Radicals and K\"othe's conjecture for skew PBW extensions, {\em Commun. Math. Stat.} {\bf 9}(2) (2021) 119--138.

\bibitem{ReyesSuarezYesica2018}A. Reyes and Y. Su\'arez, On the ACCP in skew Poincar\'e-Birkhoff-Witt extensions, {\em Beitr. Algebra Geom.} {\bf59}(4) (2018) 625--643.

\bibitem{Sarathetal1979} {B. Sarath and K. Varadarajan}, Dual Goldie dimension II, {\em Comm. Algebra} {\bf 7}(17) (1979) 1885--1899. 

\bibitem{Seiler2010} W. M. Seiler {\em Involution. The Formal Theory of Differential Equations and its Applications in Computer Algebra}. Algorithms Computat. Math, Vol. 24, Springer (2010).

\bibitem{Shock1972}R. Shock, Polynomial rings over finite dimensional rings, {\em Pacific J. Math.} {\bf42}(1) (1972) 251--257.

\bibitem{SuarezChaconReyes2021}H. Su\'arez, A. Chac\'on and A. Reyes, On NI and NJ skew PBW extensions, {\em Comm. Algebra} {\bf50}(8) (2022) 3261--3275.

\bibitem{Varadarajan1979} {K. Varadarajan,} Dual Goldie dimension, {\em Comm. Algebra} {\bf 7}(6) (1979) 565--610. 

\bibitem{Zhedanov} A. S. Zhedanov, \textquotedblleft Hidden symmetry\textquotedblright\ of Askey--Wilson polynomials, {\em Theoret. and Math. Phys.} {\bf 89}(2) (1991) 1146--1157.
\end{thebibliography}
\end{document}